\documentclass[10pt]{article}
\usepackage{amsmath}
\usepackage{amsthm}
\usepackage{amssymb}
\usepackage{epsfig}
\usepackage{psfrag}
\usepackage{graphicx}
\usepackage{caption2}
\usepackage{textcomp}
\usepackage{color}

\newcommand{\Rz}{\mathbb{R}}
\newcommand{\Nz}{\mathbb{N}}

\newcommand{\ove}{\overline}
\newcommand{\Y}{{\cal Y}^\nu}

\newcommand{\D}{{\cal D}}
\newcommand{\F}{{\cal F}}
\newcommand{\J}{{\cal J}}
\newcommand{\I}{{\cal I}}
\newcommand{\A}{{\cal A}_\nu}
\newcommand{\G}{{\cal G}}

\newcommand{\U}{{\cal U}}

\renewcommand{\S}{{\cal S}}

\newcommand{\C}{{\mathbb C}}
\newcommand{\Rzn}{\Rz^{3 \times 3}_{\,\text{\rm sym}}}
\newcommand{\Rzd}{\Rz^{3 \times 3}_{\, \text{\rm dev}}}
\newcommand{\GammaD}{\Gamma_{\text{\rm Dir}}}
\newcommand{\GammaT}{\Gamma_{\text{\rm tr}}}
\newcommand{\uD}{u^{ \text{\rm Dir}}}
\newcommand{\uDi}{u^{ \text{\rm Dir}, i}}
\newcommand{\Diss}{\text{\rm Diss}}
\newcommand{\Var}{\text{\rm Var}}
\newcommand{\tr}{\text{tr}}
\newcommand{\CC}{{\cal C}}
\renewcommand{\H}{{\cal H}}
\newcommand{\W}{{\cal W}}
\newcommand{\epsi}{\varepsilon}
\newcommand{\lan}{\langle}
\newcommand{\ran}{\rangle}

\newcommand{\argmin}{\mathop{\mathrm{Arg\,Min}}}

\newcommand{\AAAS}{}
\newcommand{\AAAE}{}
\newcommand{\UUUS}{}
\newcommand{\UUUE}{}
\newcommand{\rosso}{}
\newcommand{\azzurro}{}
\newcommand{\nero}{} 

\newtheorem{theorem}{Theorem}[section]

\newtheorem{corollary}[theorem]{Corollary}

\newtheorem{lemma}[theorem]{Lemma}
\begin{document}
                                %
                                %
                                %
\title{\Large A rate-independent model for the isothermal \\
quasi-static evolution 
of shape-memory materials
\thanks{\AAAS This research was partially supperted by the European Union via HPRN-CT-2002-00284 \textit{Smart Systems: New Materials, Adaptive
Systems and their Nonlinearities} and by the DFG Research Center 86 {\sc Matheon}
under subproject C18.\AAAE } } 
\author{\normalsize Ferdinando Auricchio\\
{\normalsize Dipartimento di Meccanica Strutturale, Universit\`a di Pavia and}\\ 
{\normalsize Istituto di Matematica Applicata e Tecnologie Informatiche - CNR}\\ 
{\normalsize via Ferrata 1, 27100 Pavia, Italy}\\
{\footnotesize e-mail: {\tt auricchio\,@\,unipv.it}}\\
\\
{\normalsize Alexander Mielke}\\
{\normalsize Weierstra\ss-Institut f\"ur Angewandte Analysis und Stochastik}\\
{\normalsize Mohrenstra\ss e 39, 10117 Berlin, Germany and }\\
{\normalsize Institut f\"ur Mathematik, Humboldt-Universit\"at zu Berlin}\\{\normalsize Rudower Chaussee 25, 12489 Berlin, Germany}\\
{\footnotesize e-mail: {\tt mielke\,@\,wias-berlin.de}}\\
\\
{\normalsize Ulisse Stefanelli}\\
{\normalsize Istituto di Matematica Applicata e Tecnologie Informatiche - CNR}\\ 
{\normalsize via Ferrata 1, 27100 Pavia, Italy}\\
{\footnotesize e-mail: {\tt ulisse\rosso .stefanelli\nero\,@\,imati.cnr.it}}}
\date{}     
                                %
                                %
\maketitle
\begin{abstract}
This note addresses a three-dimensional model for isothermal stress-induced
transformation in shape-memory polycrystalline materials. We \AAAS treat \AAAE
the problem within the framework of the \AAAS energetic formulation \AAAE of rate-independent processes and investigate existence and continuous dependence issues at both the constitutive relation and quasi-static evolution level. Moreover, we focus on time and space approximation as well as on regularization and parameter asymptotics.
\vskip3mm  

\end{abstract}
\pagestyle{myheadings}
\markright{}
                                %
                                %
                                %

\section{Introduction}\label{intro}
\setcounter{equation}{0}

Shape-memory materials are metallic alloys showing some surprising
thermo-mechanical behavior: severely deformed specimens with residual strain
up to 15\% regain their original shape after a thermal cycle ({\it
  shape-memory effect}). Moreover, the same materials are {\it super-elastic}
(also called {\it pseudo-elastic}), namely, they recover comparably large
deformations during mechanical loading-unloading cycles at prescribed
temperatures (see, among others,
\cite{Antman87,Auricchio97,Fremond-Miyazaki96,Funakubo87,Huo94,Mueller95,Wayman-Duerig90}).
These features\AAAS, which \AAAE are not present (at least to this extent) in
materials traditionally used in engineering\AAAS, \AAAE 
are at the basis of the innovative and commercially valuable applications of shape-memory materials. Namely, shape-memory technologies are nowadays exploited in a variety of different applicative contexts ranging from sensors and actuators (even microscopical), to robotics, to clamping and fixation devices, to space applications (grippers, positioners), to damping devices (shock absorption) \cite{Vanhumbeeck01}. The largest commercial success of shape-memory materials is however related to {biomedical applications}. The combination of good bio-compatibility and interesting material properties creates unique materials for medical tools and devices. Nowadays, shape-memory materials are successfully used in orthodontics (archwires), orthopedics (bone anchors, intromedullary fixations, bone staples), medical instruments, minimal invasive surgery technology (catheters, endoguidewires, grippers, cutters), drug delivery systems, and both intravascular (cardiovascular stenting, bronchial biliary, aortic aneurysm, carotid stenosis) and extravascular scaffolding. In particular, shape-memory stents are the key tool in order to implement a variety of quite successful non-invasive surgical techniques \cite{Duerig99,Stoeckel01,Stoeckel02}.

The present analysis is concerned with the quasi-static evolution of
shape-memory materials in the small-strain regime. In particular, we shall
study a macroscopic phenomenological model for shape-memory polycrystalline
materials undergoing stress-induced transformations \AAAS that \AAAE  was originally proposed by {\sc Souza et al.} \cite{Souza98} and later addressed and extended by {\sc Auricchio \& Petrini} \cite{Auricchio-Petrini02,Auricchio-Petrini04}, and {\sc Auricchio et al.} \cite{Auricchio05}. Our aim is to focus on the isothermal situation at suitably high temperatures in order to capture the super-elastic material behavior. The understanding and the efficient description of the super-elastic regime is clearly of a great applicative interest. In particular, most of the biomedical applications enlisted above are based on super-elastic deployment {\it in situ} and/or super-elastic kink resistance of shape-memory materials. 

Let us briefly recall here the basic features of the proposed model, the interested reader is of course referred to the above-mentioned contributions for all the necessary modeling details and motivations as well as for some computations and validation. The formal character of this introduction is intended to serve for the purpose of a general overview on the model and our results. In particular, (most of) the mathematical details are here omitted and will be provided in the forthcoming sections.

 Moving into the frame of Generalized Standard Materials (see {\sc Maugin} \cite{Maugin92}) and within the small-strain regime, we additively decompose the linearized deformation $\, \epsi= (\epsi_{ij})= (u_{i,j} + u_{j,i})/2 $, ($ u \,$ being the displacement from a fixed reference configuration $\, \Omega \subset \Rz^3$) into the elastic part $\, \epsi_{\text{\rm el}}\,$ and the inelastic (or transformation) part $\, z \,$ as 
\begin{equation}
  \label{additive}
  \epsi = \epsi_{\text{\rm el}} + z.
\end{equation}
At the microscopic level the super-elastic effect is interpreted as the result of a structural phase transition between different configurations of the material lattices, namely the {\em parent phase} (austenite and twinned martensite) and its shared counterpart termed {\em product phase} (detwinned martensite). In particular, the internal variable $\, z \,$ is assumed to be descriptive of the mechanical (tensorial) effect of the detwinning observed in the material.


Denoting by $\, W(\epsi,z)\,$ the stored energy density of the system, the evolution of the material will be described by the following classical relations
\begin{eqnarray}
  \sigma &=& {\partial W}/{\partial \epsi}, \label{one22}\\
-\xi&=& {\partial W}/{\partial z}, \label{two22}\\
\dot z &=& {\nabla D^*(\xi)}. \label{three22} 
\end{eqnarray} 
Here, $\, \xi \,$ denotes the thermodynamic force associated with $\, z \,$ and \eqref{three22} is the flow rule for $\, z \,$ where $\, D^* \,$ stands for the Legendre conjugate of the dissipation density $\, D \,$ (see below).
   
The material constitutive relations \eqref{one22}-\eqref{three22} may be conveniently rewritten in the following equivalent subdifferential formulation
\begin{equation}
  \label{evol}
 \binom{0}{\partial D(\dot z)} + \binom{\partial_\epsi W(\epsi,z)}{\partial_z W(\epsi,z)} \ni \binom{ \sigma}{0}. 
\end{equation}
where $\, D \,$ stands for the dissipation density and the symbol $\, \partial \,$ denotes subdifferentials in the sense of Convex Analysis (see below).

The evolution problem \eqref{evol} may be set within the frame of energetic formulations of rate-independent processes recently proposed by {\sc Mielke et al.} \cite{Mainik-Mielke05,Mielke-Theil99,Mielke-et-al02}. The notion of energetic solution (discussed in some detail in the forthcoming Section \ref{math}) is based on equivalently recasting the subdifferential problem \eqref{evol} as the coupling of a global stability condition and an energy conservation relation. In particular, the subdifferential relation \eqref{evol} is rewritten as
\begin{eqnarray}
\text{\rm\bf (stability)}&&\!\!\!\!\!\!\!\!  (\epsi(t),z(t)) \in
\argmin_{(\overline \epsi,\overline z)}  \Big(  W(\overline \epsi,\overline z)
-\sigma(t):\overline \epsi + D(\AAAS{} \overline z - z(t) \AAAE ) \Big)\label{stability_intro}\\ 
\text{\rm\bf (energy equality)}&&\!\!\!\!\!\!\!\! W(\epsi(t),z(t)) - \sigma(t):\epsi(t) + \Diss_D(z,[0,t])\nonumber\\
&\qquad&\qquad \quad =  W(\epsi_0,z_0) - \sigma(0):\epsi_0 - \int_0^t \dot \sigma(s):\epsi(s) \, ds,\label{energy_intro}
\end{eqnarray}
for all $\, t \geq 0$. Here, we assume to be given some suitable initial data $\, (\epsi_0,z_0)\,$ and the stress $\, t \mapsto  \sigma(t) \,$ and denote the total dissipation of the system on $\, [0,t]\,$ as
$$\Diss_D(z,[0,t]):=\sup\left\{\sum_{i=1}^N D(z(t_i) - z(t_{i-1})) \ : \ \{0=t_0<t_1<\dots<t_{N-1}<t_N=t\}\right\},$$
where the supremum is taken with respect to all finite partitions of $\, [0,t]$.
Energetic formulations were originally developed for shape-memory alloys in
{\sc Mielke \& Theil} and {\sc Mielke et al.}
\cite{Mielke-Theil99,Mielke-Theil04,Mielke-et-al02}, and have shown to be
extremely well-suited for a variety of different rate-independent
situations. In particular, they \AAAS have \AAAE been successfully considered in connection with elasto-plasticity \cite{DalMaso05,DalMaso07,Mielke02,Mielke03,Mielke03b,Mielke04}, damage \cite{Mielke-Roubicek05}, brittle fractures \cite{DalMaso04}, delamination \cite{Mainik-Mielke05}, ferro-electricity \cite{Mielke-Timofte05}, shape-memory alloys \cite{Mielke-Roubicek03,Mielke-Theil99,Mielke-et-al02}, and vortex pinning in superconductors \cite{Schmid-Mielke05}. The reader is referred to {\sc Mielke} \cite{Mielke05} for a comprehensive survey of the mathematical theory.

Let us now introduce the precise form of $\, W \,$ we will deal with. Namely, we choose
\begin{equation}
  \label{W}
  W(\epsi,z) = \frac12 \C(\epsi - z): (\epsi -z) + c_1 |z| + c_2 |z|^2 + I(z) + \frac{\nu}{2}|\nabla z|^2.
\end{equation}
Here, $\, \C \,$ is the elasticity tensor and the positive parameters $\, c_1
\,$ and $\, c_2 \,$ are given. Indeed, in \cite{Souza98} the constant $\, c_1
\,$ is assumed to depend explicitly on the temperature of the specimen while
here temperature effects are neglected. On the other hand, $\, c_2 \,$
measures the occurrence of some hardening phenomenon with respect to the
internal variable $\, z$. The function $\, I \,$ is the indicator of a fixed
closed ball of radius $\, c_3 >0$. In particular, $\, c_3 \,$ represents the
maximum modulus of transformation strain that can be obtained by alignment
(detwinning) of the martensitic variants. Finally, the positive coefficient
$\, \nu \,$ is expected to measure some nonlocal interaction effect for the
internal variable $\, z\,$ and $\, \nabla z \,$ stands for the usual gradient
with respect to to spatial variables. Indeed, gradients of inelastic strains
have already been considered in the frame of shape-memory materials by {\sc
  Fr\'emond} \cite{Fremond02} and the reader is referred also to {\sc Arndt et
  al.} \cite{Arndt-Griebel-Roubicek03}, {\sc Fried \& Gurtin}
\cite{Fried-Gurtin94}, {\sc Kru\v z\'\i k et al.} \cite{Kruzik05}, {\sc Mielke
  \& Roub\' \i \v cek} \cite{Mielke-Roubicek03}, {\sc Roub\' \i \v cek}
\cite{Roubicek02,Roubicek04} for examples and discussions on nonlocal \AAAS
energy contributions of \AAAE $\, z$.

The proposed model is capable of describing the main features of the super-elastic evolution of shape-memory materials. In particular, the internal variable tensorial character of the model allows for taking into account the so-called {\it single-variant martensite reorientation} phenomenon. Namely, also in the case the material is fully transformed into product phase (i.e. $\, |z|=c_3$), inelastic strain changes can still be experienced due to variant reorientation ($\dot z \not = 0$). This fact is experimentally observed and turns out to be crucial with respect to applications. Moreover, whenever not restricted to the isothermal situation, the model turns out the be thermodynamically consistent in the sense that the Second Law of Thermodynamics is satisfied in the form of the Clausius-Duhem inequality.

As for the full quasi-static evolution of the material we shall couple the constitutive relation \eqref{evol} with the equilibrium equation
\begin{equation}\label{equi}
\text{div}\,\sigma + f =0 \quad \text{in} \ \ \Omega,
\end{equation}
where $\, f \,$ is a given body force, suitably complemented with some prescribed boundary displacement and boundary traction in distinguished parts of the boundary of $\, \Omega$.

The first issue of this paper is that of adapting the above referred abstract theory for energetic formulations to the quasi-static evolution problem and obtain that (Theorem \ref{const2})

\begin{description}
\item[\quad(existence)] \ \  the quasi-static problem admits at least one energetic solution $\, t \mapsto (u(t),z(t))$.
\end{description}

We shall be concerned with some specific regularization of the original quasi-static model. Namely, some smooth variant of the potential $\,W \,$ above turns out to be better suited for the sake of numerical considerations. In particular, we will consider a regularized version of the model by posing
\begin{equation}\label{W2}
  \rosso W_{\rho,\nu}\nero(\epsi,z)= \frac12 \C(\epsi - z): (\epsi -z) + F_\rho(z) + \frac{\nu}{2}|\nabla z|^2,
\end{equation}
where \rosso $\, \nu \geq 0\,$ and \nero $\, F_\rho \,$ is some regularization of $\, F_0: z \mapsto c_1 |z| + c_2 |z|^2 + I(z)\,$ obtained by penalization and smoothing and depending on the regularization parameter $\, \rho\geq 0$.  This regularization is exactly the starting point of {\sc Auricchio \& Petrini} \cite{Auricchio-Petrini02,Auricchio-Petrini04}, and has been exploited in {\sc Auricchio et al.} \cite{Auricchio05} as well (in all these papers $\, \nu=0\,$ though).

A second focus of the present contribution is on unique solvability of the regularized model. In particular, we check that 
\begin{description}
\item[\quad(uniqueness for $\, {\boldsymbol \rho \boldsymbol >\boldsymbol 0}$)] \ \ for $\, \rho >0 $, the quasi-static problem has a unique solution.
\end{description}
This uniqueness result was proved in an abstract frame by {\sc Mielke \& Theil} \cite{Mielke-Theil99,Mielke-Theil04} and is here reconsidered in the specific situation of the regularized version of the quasi-static problem.
 
A quite natural approach to rate-independent evolution problems relies on implicit time-discretization. This perspective is here investigated and complemented with some space approximation technique. In particular, the main novelty of this paper is the convergence analysis for the discretized-regularized model. Namely, we consider the (possibly joint) limits with respect to the time-steps $\, \tau \,$ of time partitions (here considered to be constant for simplicity), the space mesh size $\, h \,$ (conforming finite elements are exploited), and the regularization parameter $\, \rho$. In particular, denoting by $\, (u,z)_{\rho,\tau,h} \,$ the unique solution to the space-time discrete problem with the parameter-choice $\, \rho \geq 0 \,$ (time-interpolant, piecewise constant on the time-partition) and by $\, (u,z)_\rho \,$  the time-continuous solution to the problem for $\, \rho \geq 0$, we prove the following (Theorem \ref{EP})
\begin{description}
\item[\quad(convergence for $\, {\boldsymbol \rho \boldsymbol >\boldsymbol 0}$)]  \ \ for $\, \rho >0$, $\, (u,z)_{\rho,\tau,h} \,$ converges to $\, (u,z)_\rho\,$ as $\, (\tau,h)\to (0,0)$,
\item[\quad{(full convergence)}] \ \ up to a subsequence, $\, (u,z)_{\rho,\tau,h} \to (u,z)_0\ \ \text{as} \ \  (\rho,\tau,h)\to (0,0,0)$.
\end{description}
Of course the topologies under which the latter convergences hold true will be
specified in the forthcoming \AAAS sections.
 
Indeed much more is true and we are in the position of giving a full picture
of convergences for the model subsequently. 
Moving \AAAE from Section \ref{math} where the mathematical formulation of the problem is presented, we shall organize our results by successively increasing complexity. Section \ref{constitutive} addresses the analysis of the constitutive relation problem \eqref{evol}, namely the zero-dimensional problem. In particular, we prove well-posedness and convergence of time-discrete approximations. Then, the three-dimensional minimum problem arising from time-discretization is addressed in Section \ref{incremental} where we also investigate well-posedness and convergence of space approximations along with suitable error bounds. Some a priori bounds and a preliminary convergence result for the incremental solutions to the problem in case the time-partition is fixed are discussed in Section \ref{INCRE}. Finally, the three-dimensional quasi-static evolution problem is tackled in Section \ref{evoevo} where we provide the above mentioned existence, uniqueness, and convergence results for the space-time discrete solutions. Finally, Section \ref{limitrho} deals with convergence issues with respect to parameters and discretizations in full generality.

\section{Mathematical formulation}\label{math}
\setcounter{equation}{0}

\paragraph{Tensors.} We will denote by $\, \Rzn \,$ the space of symmetric $\, 3 \times 3\,$ tensors endowed with the natural scalar product $\,a : b := \text{tr}(ab)= a_{ij}b_{ij}\,$ (summation convention) and the corresponding norm $\, |a|^2:= a:a\,$ for all $\, a, \, b \in \Rzn$. The space $\, \Rzn \,$ is orthogonally decomposed as $\,\Rzn = \Rzd \oplus \Rz \,1_2$, where $\, \Rz \,1_2\,$ is the subspace spanned by the identity 2-tensor $\, 1_2\,$ and $\, \Rzd \,$ is the subspace of deviatoric symmetric $\, 3 \times 3\,$ tensors. In particular, for all $\, a \in \Rzn$, we have that $\, a = a_{\text{\rm dev}} + \tr(a)1_2/3$. For all $\, u \in H^1_{\text{loc}}(\Rz^3; \Rz^3)\,$ we let $\, \epsi(u)\in L^2_{\text{loc}}(\Rz^3; \Rzn )\,$ denote the standard symmetric gradient.

\paragraph{Reference configuration.} We shall assume $\, \Omega \,$ to be a non-empty, bounded, and connected open set in $\, \Rz^3 \,$ with a Lipschitz continuous boundary. The space dimension $\, 3 \,$ plays essentially no role throughout the analysis and we would be in the position of reformulating our results in $\, \Rz^d \,$ with no particular intricacy. We assume that the boundary $\, \partial \Omega \,$ is partitioned in two disjoint open sets $\, \GammaT \,$ and $\, \GammaD\,$ with $\, \partial \GammaT = \partial \GammaD\,$ (in $\, \partial \Omega$). We ask $\, \GammaD \,$ to be such that there exists a positive constant $\, c_0 \,$ depending on $\, \GammaD\,$ and $\, \Omega \,$ such that the Korn inequality
\begin{equation}\label{korn}
c_0\| u \|^2_{H^1(\Omega;\Rz^3)} \leq \| u \|^2_{L^2(\GammaD;\Rz^3)} + \| \epsi (u) \|^2_{L^2(\Omega; \Rzn)},
\end{equation}
holds true for all $\, u \in H^1(\Omega;\Rz^3)$. It would indeed suffice to impose $\, \GammaD \,$ to have a positive surface measure (see, e.g., \cite[Thm. 3.1, p. 110]{Duvaut-Lions}).

\paragraph{Prescribed boundary displacement.} We will prescribe some non-homogeneous Dirichlet\linebreak boundary conditions on $\, \GammaD$. To this end, we will assign $\,\uD \in C^1([0,T]; H^{1/2}(\GammaD, \Rz^3)) \,$ or, equivalently, $\, \uD \in  C^1([0,T]; H^1(\Omega,\Rz^3))\,$ whose trace on $\, \GammaD\,$ is the prescribed boundary value for the displacement $\, u$. On $\, \GammaT \,$ some time-dependent traction will be prescribed instead.

\paragraph{Elastic energy.} Let $\, \C \,$ be the elasticity tensor. The latter  is regarded as a symmetric positive definite linear map $\, \C : \Rzn \rightarrow \Rzn$. We shall assume that the orthogonal subspaces $\, \Rzd \,$ and $\, \Rz \, 1_2\,$ are invariant under $\, \C $. This amounts to say that indeed 
$$ \C a= \C_{\text{\rm dev}} a_{\text{\rm dev}} + \kappa\, \text{tr}(a)1_2,$$ 
for a given $\, \C_{\text{\rm dev}}: \Rzd \rightarrow \Rzd\,$ and a constant $\, \kappa$, and all $\, a \in \Rzn$. The case of isotropic materials is given by $\, \C_{\text{\rm dev}} = 2G (1_4 - 1_2 \otimes 1_2/3) \,$ and $\, G \,$ and $\, \kappa \,$ are respectively the shear and the bulk moduli. The latter decomposition is not exploited in our analysis but it is clearly suggested by the mechanical application.

We will make use of the stored elastic energy functional $\, \CC : L^2(\Omega; \Rzn) \rightarrow [0,+\infty) \,$ defined as
$$ \CC (a) :=\frac12 \int_\Omega \C(a):a \, dx.$$ 

\paragraph{Inelastic energy.}
As for the stored inelastic (or transformation) energy we shall prescribe the function $\, F : \Rzd \rightarrow [0,+\infty]\,$ as 
$$F(a) = c_1|a| + c_2|a|^2 + I(a),$$
where $\,I: \Rzd \rightarrow [0,+\infty]\,$ is the indicator function of the ball $\, \{a \in \Rzd \ : \ |a| \leq c_3\}\,$ and the positive constants $\, c_1, \, c_2,$ and $\, c_3 \,$ are given. Moreover, the stored inelastic energy functional is defined as 
$\, \F :  L^2(\Omega; \Rzd) \rightarrow [0,+\infty]\,$ as
$$\F(a):= \int_\Omega F(a) \, dx \quad \text{if} \ \ F(a) \in L^1(\Omega) \ \ \text{and} \ \ \F(a)=+\infty \ \ \text{otherwise}.$$
The well-posedness and time discretization issues discussed here do not rely on the particular form of $\, F \,$ and could be adapted to any uniformly convex, proper, and lower semicontinuous function. We however prefer to stick to the actual modeling choice for the sake of clarity.
In the forthcoming of the paper we will address some suitable regularization of $\, F$. Indeed, we introduce an approximation parameter $\, \rho \geq 0 \,$ and some functions 
\begin{gather}
  F_\rho \in C^{2,1}(\Rzd) \ \ \text{with}   \  \ \azzurro \nabla^2 F_\rho\nero \ \ \text{bounded}, \ \ \nabla^2F_\rho \geq c_2 1_4, \  \ \text{and $\,F_\rho(0) = 0$,} \label{F_rho}
\end{gather}
and define $\, F_0:=F$.
An example in the direction of \eqref{F_rho} is
\begin{gather}
F_\rho(a):= c_1 (\sqrt{\rho^2 + |a|^2} - \rho) + c_2|a|^2 +\varphi(|a|)/\rho\nonumber\\
 \text{for} \ \ \varphi \in C^{2,1}(\Rz), \ \ \varphi' \in L^\infty(\Rz),  \ \ \varphi''\geq0,  \ \ \varphi(r) =0 \ \ \text{iff} \ \ r \leq c_3.\label{sceltaF}
\end{gather}

Exactly as above, for all $\, \rho \geq 0 \,$ we let the regularized stored inelastic energy functional
$\, \F_\rho :  L^2(\Omega; \Rzd) \rightarrow [0,+\infty)\,$ be defined as
$$\F_\rho(a):= \int_\Omega F_\rho(a) \, dx,$$
and $\, \F_0:=\F$.
Finally, we shall be considering also some space-regularized situation. To this end, let $\, \rho, \, \nu\geq 0 \,$ and define $\, \F_{\rho,\nu}\ :  L^2(\Omega; \Rzd) \rightarrow [0,+\infty]\,$ as
$$\F_{\rho,\nu}(a):= \int_\Omega\left( F_\rho(a)  + \frac{\nu}{2} |\nabla a|^2\right) dx,$$
where $\, (\nabla a)_{ijk}= \partial a_{ij}/\partial x_k\,$ is the usual gradient in the distributional sense and $\, |\cdot| \,$ denotes here the Euclidean norm.

\paragraph{Stored energy.} Following the above introductory discussion, we define the stored (Helmholtz free) energy functional for $\, \rho,\, \nu \geq 0\,$ as
$$\W_{\rho,\nu}(u,z):=  \CC(\epsi(u) - z) + \F_{\rho,\nu}(z).$$

\paragraph{Load and traction.} We assume to be given the body force $\, f \in W^{1,1}(0,T;L^2(\Omega;\Rz^3))\,$ and a surface traction $\, g \in W^{1,1}(0,T;L^2(\GammaT;\Rz^3))$. In particular, one can define the total load $\, \ell \in W^{1,1}(0,T;(H^1(\Omega; \Rz^3))')\,$ (the prime denotes here the dual) as
$$\lan \ell(t), u \ran := \int_\Omega f \cdot u \, dx + \int_{\GammaT} g \cdot u \, d\H^{2} \quad \forall u \in H^1(\Omega; \Rz^3), \ t \in [0,T],$$
where $\, \H^2 \,$ is the 2-dimensional Hausdorff measure and $\, \lan\cdot,\cdot\ran \,$ denotes the duality pairing between $\, (H^1(\Omega;\Rz^3))'\,$ and $\, H^1(\Omega;\Rz^3)$.

\paragraph{State space.} We set our problem by letting
$$ \Y={\cal U} \times {\cal Z}^\nu:=  H^1(\Omega, \Rz^3)\times H^{j(\nu)}(\Omega;\Rzd).$$
Here $\, j(\nu) =0 \,$ for $\,\nu=0 \,$ and $\, j(\nu) =1 \,$ otherwise. For all $\, \overline u \in H^1(\Omega;\Rz^3)$, let us define $\, \Y (\overline u) \subset \Y\,$ as
$$\Y(\overline u):= \{ (u,z) \in \Y \ : \ u=\overline u \ \  \text{on} \ \ \GammaD \},$$
Then, for all $\, t \in [0,T]$, we shall define the phase space of the process as $\,\Y(\uD(t)).$ For the sake of later purposes (see also \eqref{W}) let us denote by $\, W_\rho : \Rzn \times \Rzd \rightarrow [0,+\infty)\,$ the function
$$W_\rho(\epsi,z) := \frac12\C(\epsi -z): (\epsi - z) + F_\rho(z),$$
\rosso 
and remark that, owing to \eqref{F_rho}, the second derivative $\, \nabla^2 W_\rho\,$ is bounded and Lipschitz continuous. Moreover, we denote
\nero 
by $\, \A: \Y \rightarrow [0,\infty)\,$ the quadratic form
$$\A(u,z):= \CC(\epsi(u) - z) + c_2\int_\Omega |z|^2dx + \frac{\nu}{2}\int_\Omega|\nabla z|^2dx\quad \forall (u,z)\in \Y$$
and by $\, \alpha >0 \,$ the corresponding uniform ellipticity constant (depending on $\, \C, \, c_2$, and $\, \nu$).

\paragraph{Dissipation potential.} The quasi-static evolution of the material is described by means of an appropriate dissipation mechanism, see \eqref{evol}. To this aim, we choose the dissipation (pseudo)-potential $\, D: \Rzd \rightarrow [0,+\infty)\,$ to be lower semi-continuous, positively $\,1-$homogeneous, and to fulfill the triangle inequality
\begin{equation}
  \label{triangle}
  D(a) \leq D(b) + D(c)\ \ \ \text{whenever} \  \ \ a= b+c.
\end{equation}
 Moreover, we ask for some constant $\, c_D>0 \,$ such that 
$$c_D|a| \leq D(a) \quad \forall a \in \Rzd.$$
Under the current assumptions on $\, D$, the latter non-degeneracy condition is indeed equivalent to the fact that the set $\, \{a\, : \, D(a) \leq 1 \}\,$ is bounded or that $\, D \,$ does not vanish except in $\, 0$. Let us stress that $\, D \,$ turns out to be convex (see \eqref{triangle}) and that there exists a second constant $\, C_D>0 \,$ such that 
$$D(a) \leq C_D |a| \quad \forall a \in \Rzd.$$
We define the corresponding dissipation functional $\, \D: L^1(\Omega;\Rzd) \rightarrow [0,+\infty)\,$ as
$$\D(a)=\int_\Omega D(a)\, dx. $$
One shall stress that indeed, since $\, \D \,$ is obviously positively 1-homogeneous, a rate-independent evolution follows.
Moreover, we recall here that, for all $\, z:[0,T]\rightarrow \Rzd$, we let
\begin{equation}
\Diss_D(z,[s,t]):=\sup\left\{\sum_{i=1}^N D(z(t_i) - z(t_{i-1})) \ : \ \{s=t_0<t_1<\dots<t_{N-1}<t_N=t\}\right\},\label{diss}
\end{equation}
the supremum being chosen on the set of all finite partitions of $\, [s,t]\subset [0,T]$. Finally the analogous notion $\, \Diss_\D(z,[s,t])\,$ will be used for functions which take values in $\, L^1(\Omega;\Rzd)$.

\paragraph{State space approximation.} Henceforth we will be interested in some space approximation procedure. Indeed, we assume to be given a suitable sequence of approximating closed subspaces $\, \Y_h:={\cal U}_h \times {\cal Z}_h^\nu\subset \Y \,$ depending on some parameter $\, h > 0\,$ which is intended to go to zero in the limit. We shall collect and comment here the abstract assumptions which will be exploited in the following. Of course the main application we have in mind are conforming finite elements on a shape regular and quasi-optimal mesh \cite{Ciarlet78} with size $\, h\,$ on the polyhedral domain $\, \Omega$. We will firstly ask $\, \Y_h \,$ to be non-decreasing and such that $\, \cup_{h>0}\Y_h \,$ is dense in $\, \Y$. Moreover, we restrict from the very beginning to the special case when $\, {\cal Y}^0_h\equiv {\cal Y}^1_h \subset {\cal Y}^1$.

Now let $\, p_h^\nu: \Y \rightarrow \Y_h \,$ the Galerkin projector corresponding to the scalar product induced by the quadratic form $\,\A$. In particular, by introducing the bilinear form $\, {\cal B}_\nu: \Y \times \Y \to \Rz\,$ defined by
\begin{gather}
  {\cal B}_\nu \big((u_1,z_1),(u_2,z_2)\big) := \frac12 \int_\Omega \C(\epsi(u_1) - z_1):(\epsi(u_2) - z_2) + c_2 \int_\Omega z_1 \, z_2 +\frac{\nu}{2}\int_\Omega \nabla z_1 \cdot \nabla z_2\nonumber
\end{gather}
for $\,  (u_1,z_1), \, (u_2,z_2) \in \Y$, we have that, for all $\, (u,z)\in \Y$, the projection $\, p_h^\nu(u,z)\,$ may be uniquely determined by
\begin{gather}
  {\cal B}_\nu \big((u,z) -  p_h^\nu(u,z), (u_h,z_h)\big) =0 \quad \forall (u_h,z_h)\in \Y_h.\label{gal}
\end{gather}
Namely, one has that
\begin{gather}
  \A( p_h^\nu(u,z))=  {\cal B}_\nu( p_h^\nu(u,z), p_h^\nu(u,z))\leq \A(u,z) \quad \forall (u,z) \in \Y.\label{gal2}
\end{gather}
Let us explicitly observe that $\, p_h^\nu \,$ is pointwise converging in $\, \Y\,$ to the identity as $\, h \rightarrow 0 $.

Next, let us introduce a pair of operators $\, q_h: {\cal U}\rightarrow {\cal U}_h \,$ and $\, r_h^\nu:{\cal Z}^\nu\rightarrow {\cal Z}^\nu_h\,$ and ask them to be pointwise converging to the identity as $\, h \rightarrow 0$. More specifically, we will ask for 
$$h \to 0, \ \nu \to 0 \ \ \Rightarrow \ \ r^\nu_h(z) \rightarrow z \quad \forall z \in {\cal Z}^\nu.$$
Moreover, we require that
\begin{equation}
  \label{clem}
  z \in {\cal Z}^0 \ \ \text{and} \ \ |z| \leq c_3 \ \ \text{a.e. in} \ \ \Omega \quad \Rightarrow \ \ |r_h^\nu(z)| \leq c_3 \ \ \text{a.e. in} \ \ \Omega , 
\end{equation}
and that $\, r_h^0 : {\cal Z}^1 \rightarrow {\cal Z}^1\,$ maps bounded sets into bounded sets. As for $\, r_h^\nu \,$ an example of operator fulfilling the assumptions is the component-wise Cl\'ement interpolant from $\, L^1(\Omega;\Rzd) \,$ to the space of piecewise linear functions \cite{Clement74}. In this case, relation \eqref{clem} follows from Jensen's inequality.

\section{Analysis of the constitutive relation}\label{constitutive}
\setcounter{equation}{0}

Let us start our analysis by focusing on the constitutive material relation. Namely, we neglect for the moment the coupling of the material model with the equilibrium problem \eqref{equi}. Assuming to be given a tension history, we solve for the elastic and the inelastic strain starting from a given state. The understanding of this simplified (reduced) problem will be crucial. First of all, a detailed study of the constitutive relation is surely an important step in the direction of the investigation of the full quasi-static evolution problem. This in especially true with respect to numerics. Indeed, the efficient solution of the constitutive relation is the key ingredient for a full discretization procedure. Secondly, the full equilibrium system might reduce to a zero-dimensional problem under specific yet common geometric restrictions or symmetries. Finally, we aim to give in this somehow (notationally) simplified situation the main points of our analysis.

Assuming to be given $\,  \sigma: [0,T]\rightarrow \Rzn$, we shall determine $\, \epsi : [0,T] \rightarrow \Rzn\,$ and $\, z:[0,T] \rightarrow \Rzd\,$ starting from $\, (\epsi_0,z_0)\,$ and fulfilling \eqref{evol}. Of course, since the transformation strain $\, z \,$ is assumed to be deviatoric and the elasticity tensor $\, \C \,$ decomposes as above, the problem could be easily reformulated in the deviatoric subspace $\, \Rzd \,$ only. We however prefer not to exploit this simplification for the sake of consistency with the forthcoming analysis.

Let $\, \rho \geq 0 \,$ be fixed throughout this section. We shall be concerned with the energy function $\, W_\rho(\epsi,z) - \sigma(t) : \epsi\,$ which is defined for all $\, (t, \epsi,z) \in [0,T]\times\Rzn\times \Rzd.$
Moreover, let us define the set of {\it stable states} at time $\, t \in [0,T]\,$ as
\begin{gather}
  S(t):=\Big\{(\epsi,z) \in \Rzn \times \Rzd  \ \ \text{such that}, \ \ \forall (\overline \epsi, \overline z) \in  \Rzn \times \Rzd,\nonumber\\
W_\rho(\epsi,z) - \sigma(t) : \epsi \leq W_\rho(\overline \epsi, \overline z) - \sigma(t) : \overline \epsi+ D(\overline z - z)  
\Big\},  \label{stable}
\end{gather}
and $\, \S:= \cup_{t \in [0,T]}(t,S(t))$.

As for an {\it energetic solution} of \eqref{evol} we mean a pair $\, (\epsi,z) : [0,T] \rightarrow \Rzn\times \Rzd\,$ such that the function $\, t \mapsto \dot \sigma(t): \epsi(t)\,$ is integrable and, for all $\, t \in [0,T]$, 
\begin{eqnarray}
&&  (\epsi(t),z(t))\in S(t), \label{s}\\
&&W_\rho(\epsi(t),z(t)) - \sigma (t) : \epsi(t) +\Diss_D(z,[0,t]) \nonumber\\
&&\qquad = W_\rho(\epsi_0,z_0) - \sigma (0) : \epsi_0- \int_0^t \dot\sigma(s):\epsi(s)\, ds.\label{e}
\end{eqnarray}

Let us now comment on the equivalence between \eqref{evol} and the {\it energetic formulation} \eqref{s}-\eqref{e}. To this end we will focus for simplicity on the smooth case $\, \rho>0$. Indeed, the argument for the situation $\, \rho =0\,$ is just slightly less straightforward from a notational viewpoint. Using the definition of the subdifferential $\, \partial D(\dot z)$, relation \eqref{evol} turns out to be equivalent to  
\begin{gather}
  (\partial_\epsi W_\rho(\epsi,z) - \sigma):(v- \dot \epsi) + \partial_zW_\rho(\epsi,z):(w - \dot z)+ D(w) - D(\dot z)\geq 0 \nonumber\\
 \forall (v,w) \in \Rzn \times \Rzd, \ \ \text{a.e. in} \ \ (0,T).\label{evol2}
\end{gather}
Now, by respectively choosing $\, (v,w)=(k \overline v, k \overline w)\,$ and letting $\, k \rightarrow +\infty \,$ or $\,  (v,w)=(0,0)\,$ in the latter relation we easily get that
\begin{eqnarray}
 &&(\partial_\epsi W_\rho(\epsi,z) - \sigma):\overline v + \partial_zW_\rho(\epsi,z):\overline w + D(\overline w) \geq 0 \nonumber\\
 && \qquad\qquad\qquad\qquad\qquad\qquad\forall (\overline v, \overline w) \in \Rzn \times \Rzd, \ \ \text{a.e. in} \ \ (0,T),\label{s1}\\
&&(\partial_\epsi W_\rho(\epsi,z) - \sigma):\dot \epsi + \partial_zW_\rho(\epsi,z):\dot z + D(\dot z) \leq 0 \quad \text{a.e. in} \ \ (0,T).\label{e1}
\end{eqnarray}
Of course \eqref{evol2} and \eqref{s1}-\eqref{e1} are equivalent. Now, since $\, W_\rho \,$ is strictly convex, we have that $\, (\epsi(t),z(t) )\,$ is the almost everywhere unique minimizer of 
$$(\overline \epsi,\overline z) \mapsto W_\rho(\overline \epsi,\overline z) -
\sigma: \overline \epsi + 
D(\overline z - z\AAAS (t) \AAAE ).$$
In particular, by assuming $\, \epsi, \, z$, and $\, \sigma \,$ to be absolutely continuous (see below), we readily check that \eqref{s} holds. Moreover \eqref{s1}-\eqref{e1} imply that 
$$(\partial_\epsi W_\rho(\epsi,z) - \sigma):\dot \epsi + \partial_zW_\rho(\epsi,z):\dot z + D(\dot z) = 0 \quad \text{a.e. in} \ \ (0,T),$$
which can be rewritten as
$$\frac{d}{dt}\big(W_\rho(\epsi,z) - \sigma: \epsi \big)= - \dot \sigma:\epsi - D(\dot z)\quad \text{a.e. in} \ \ (0,T).$$
Hence, by integrating the latter on $\,(0,t)\,$ for $\, t \in [0,T]$, we readily deduce \eqref{e}. Vice versa, \eqref{e} allows us to recover \eqref{s1}-\eqref{e1}\, at once by differentiating and exploiting \eqref{s}.

The main advantage of the energetic formulation \eqref{s}-\eqref{e} is that it
\AAAS does 
involve \AAAE neither derivatives of constitutive quantities nor of the solution. It is hence particularly well-suited for the aim of proving well-posedness results and it simply generalizes to possibly non-convex situations. 

The aim of this section is to exploit here the abstract existence theory for energetic formulations developed in \cite{Francfort-Mielke05,Mainik-Mielke05} and adapt it to the current modeling situation. 

\paragraph{The incremental problem.} In order to find an energetic solution to \eqref{s}-\eqref{e} we shall consider an implicit time discretization procedure. At first, let us observe that, for all $\, \overline z \in \Rzd \,$ and $\, t \in [0,T]$, the function $\, (\epsi,z) \mapsto W_\rho(\epsi,z)-\sigma(t):\epsi +D(z - \overline z) \,$ has a unique minimum since it is uniformly convex and coercive. Let now the partition $\, P:=\{0=t_0<t_1<\dots<t_{N-1}<t_N=T\}\,$ be given with diameter $\, \tau = \max_{i=1, \dots, N}t_i- t_{i-1}$. Moreover, let $\, (\epsi_0, z_0) \in S(0)\,$ be a given initial datum. One should consider that, for any given $\, z_0 \in \Rzd$, there exists a unique $\, \epsi_0 = {\cal L} z_0$, where $\, {\cal L}=id\,$ here, with $\, (\epsi_0, z_0) \in S(0)$. Hence, we solve iteratively the minimum problem
\begin{equation}
  \label{return}
 (\epsi_i,z_i) \in \argmin_{(\epsi,z) \in \Rzn \times \Rzd}\big(W_\rho(\epsi,z) - \sigma(t_i):\epsi +D(z - z_{i-1}) \big) \quad \text{for} \ \ i=1, \dots,N.
\end{equation}
 We shall refer to the latter as the {\it incremental problem} associated with \eqref{s}-\eqref{e}. Let us explicitly observe that, by the triangle inequality, any solution $\, (\epsi_i,z_i)\,$ to \eqref{return} solves also 
 \begin{equation}
  \label{return2}
  (\epsi_i,z_i) \in \argmin_{(\epsi,z) \in \Rzn \times \Rzd}\big(W_\rho(\epsi,z)  -  \sigma(t_i):\epsi  +D(z - z_{i})\big)\quad \text{for} \ \ i=1, \dots,N.
\end{equation}

\paragraph{Error propagation.} We shall start by providing a continuous dependence result for the single-step minimum problem in \eqref{return}. Referring to the forthcoming time-stepping procedure, the following estimate can be seen as some error propagation control.
\begin{lemma}[Continuous dependence]\label{cont_dep_const} Let $\, (\sigma^j, \overline z^j) \in \Rzn \times \Rzd \,$ $\, j=1,2,$ be given and $\, (\epsi^j,z^j):=\argmin_{( \epsi,  z) \in \Rzn \times \Rzd}(W_\rho(\epsi,z) - \sigma^j:\epsi + D(z - \overline z^j))$. Then 
  \begin{equation}
    \label{cont_dep_const_eq}
    |\epsi^1 - \epsi^2|^2 + |z^1 - z^2|^2 \leq \frac{1}{\alpha^2}|\sigma^1 - \sigma^2|^2 + \frac{4}{\alpha}D(\overline z^1 - \overline z^2).
  \end{equation}
\end{lemma}

\begin{proof}
  Since $\, (\epsi^1,z^1)\,$ is minimal and $\, W_\rho \,$ is uniformly convex of constant $\, \alpha\,$ one has that
  \begin{eqnarray}
    \alpha |\epsi^1 - \epsi^2|^2 + \alpha |z^1 - z^2|^2 &\leq& W_\rho(\epsi^2,z^2)- \sigma^1:\epsi^2 + D(z^2 - \overline z^1)\nonumber\\
&-&W_\rho(\epsi^1,z^1)+\sigma^1:\epsi^1 -D(z^1 - \overline z^1).\nonumber
 \end{eqnarray}
On the other hand, the minimality of $\,(\epsi^2,z^2)\,$ entails that
 \begin{gather}
   0 \leq W_\rho(\epsi^1,z^1)- \sigma^2:\epsi^1 + D(z^1 - \overline z^2)
-W_\rho(\epsi^2,z^2)+\sigma^2:\epsi^2 -D(z^2 - \overline z^2).\nonumber
 \end{gather}
Taking the sum of the latter relations and exploiting the triangle inequality \eqref{triangle} we get that 
$$ \alpha |\epsi^1 - \epsi^2|^2 + \alpha |z^1 - z^2|^2 \leq (\sigma^1 - \sigma^2):(\epsi^1- \epsi^2) + 2D(\overline z_1 - \overline z_2),$$
whence the assertion follows.
\end{proof}

\paragraph{The evolution problem.} We shall now provide the main result of this section which follows by passing to the limit in the above described time-discrete approximation.

\begin{theorem}[Existence for $\,\boldsymbol \rho \boldsymbol \geq \boldsymbol 0$]\label{const}
Given $\, \sigma \in W^{1,1}(0,T;\Rzn) \,$ and  $\, (\epsi_0,z_0)\in S(0) \,$ there exists an energetic solution $\, (\epsi,z)\,$ to \eqref{s}-\eqref{e} such that $\, (\epsi(0),z(0))=(\epsi_0,z_0)$. Moreover $\, (\epsi,z) \in W^{1,1}(0,T;\Rzn\times \Rzd)$.
\end{theorem}

\begin{proof}
  Let us choose a sequence of partitions $\, P^n:=\{0=t_0^n<t_i^n<\dots<t_{N^n-1}^n<t_{N^n}^n=T\} \,$ with diameters $\, \tau^n=\max_{i=1,\dots,N^n}(t_i^n - t_{i-1}^n) \,$ going to zero. Owing to the above discussion, we uniquely determine a sequence of solutions $\, \{(\epsi_i^n,z_i^n)\}_{i=0}^{N^n}\,$ to the corresponding incremental problems \eqref{return} such that $\, (\epsi^n_0,z^n_0)=(\epsi_0,z_0)$.
We shall denote by  $\, (\epsi^n , z^n)\,$ the {\it incremental solution}, i.e. the right-continuous piecewise-constant interpolant of  $\, \{(\epsi_i^n,z_i^n)\}_{i=0}^{N^n}\,$ on the partition $\, P^n$, and by $\, \tau^n, s^n:[0,T] \rightarrow [0,T] \,$ the functions $\, \tau^n(t):= t^n_i\,$ for $\, t \in (t^n_{i-1},t^n_i]$, and $\, s^n(t):=t^n_{i-1} \,$  for $\, t \in [t^n_{i-1},t^n_i)$,$\, i=1 \, \dots, N^n$.

Since $\,\{ (\epsi_i^n,z_i^n)\}_{i=0}^{N^n}\,$ solves \eqref{return2} with $\, z_i^n \,$ replacing $\, z_i$, one directly gets that $\, (\epsi_i^n,z_i^n) \in S(t_i^n)\,$ for all $\, i=1, \dots, N^n$. Moreover, from \eqref{return} and the minimality of $\, (\epsi^n_i,z^n_i)$, we compute that
\begin{gather}
 W_\rho( \epsi_i^n,z_i^n) -  \sigma(t_i^n): \epsi_i^n - W_\rho( \epsi_{i-1}^n,z_{i-1}^n) + \sigma(t^n_{i-1}): \epsi^n_{i-1}\nonumber\\
 + D(z_i^n - z_{i-1}^n)
 \leq  - (\sigma(t_i^n)-\sigma(t_{i-1}^n)): \epsi_{i-1}^n.\nonumber
\end{gather}
Next, taking the sum of the latter relation for $\, i = 1, \dots, m\,$ and $\, m \leq N^n$, we get that
\begin{gather}
 W_\rho( \epsi_m^n,z_m^n) -  \sigma(t_m^n): \epsi_m^n - W_\rho( \epsi_{0},z_{0}) + \sigma(0): \epsi_{0}\nonumber\\
 + \sum_{i=1}^m D(z_i^n - z_{i-1}^n)
 \leq - \int_{0}^{t^n_m} \dot \sigma:  \epsi^n \, ds.\label{poi}
\end{gather}
 Hence, it suffices to apply the discrete Gronwall lemma and exploit the coercivity of $\, W_\rho \,$ in order to check that
\begin{equation}
  \label{bound2}
  \sup_{t \in [0,T]} W_\rho( \epsi^n(t),  z^n(t)) \ \ \text{and} \ \  \Diss_D( z^n,[0,T]) \ \ \text{are bounded independently of} \ \ n.
\end{equation}
Indeed, the latter bound depends on $\, W_\rho(\epsi_0,z_0) \,$ and $\, \| \sigma\|_{W^{1,1}(0,T;\Rzn)}\,$ only. 

In order to pass to the limit with $\, n \,$ we exploit Helly's selection principle and find a (not relabeled) subsequence of partitions and a non-decreasing function $\, \phi:[0,T] \rightarrow [0,+\infty)\,$ such that
\begin{gather}
   z^n(t) \rightarrow z(t), \quad
 \Diss_D( z^n,[0,t])\rightarrow \phi(t) \quad \text{for all} \ \ t \in [0,T],\label{propf}\\
\text{and} \ \ \Diss_D(z,[s,t]) \leq \phi(t) - \phi(s) \quad \forall [s,t]\subset [0,T].\label{propfi}
\end{gather}
Consequently, for all $\, t \in [0,T]$, we readily find the unique limit $\, \epsi(t)= {\cal L} z(t)\,$ since $\, \epsi^n(t)={\cal L}z_n(t) \rightarrow {\cal L} z(t)$.

Next, we check that $\, \S \,$ is closed. Indeed, let the sequence $\, (t_k,\epsi_k,z_k) \in \S \,$ converge to $\, (t,\epsi,z)\,$ in $\, [0,T]\times \Rzn \times \Rzd$. Then, since $\, W_\rho \,$ is lower semicontinuous and $\, \sigma \,$ is continuous, for all $\, (\overline \epsi, \overline z) \in \Rzn \times \Rzd$,
\begin{gather}
 W_\rho( \epsi,z) - \sigma(t):\epsi \leq \liminf_{k\rightarrow +\infty}\big(W_\rho( \epsi_k,z_k) - \sigma(t_k) : \epsi_k\big) \nonumber\\
\leq \liminf_{k\rightarrow +\infty} \big(W_\rho(\overline \epsi, \overline z) -\sigma(t_k) : \overline \epsi  + D(\overline z - z_k)\big) =W_\rho(\overline \epsi, \overline z)- \sigma(t):\overline \epsi + D(\overline z - z).\nonumber
\end{gather}
Namely $\, (t,\epsi,z )\in \S$. We shall exploit the latter closure property in order to prove that $\, (\epsi(t),z(t)) \,$ is a stable state. Indeed, recalling that $\, t \in [0,T]\,$ is fixed, one readily checks that the sequence $\, \tau^n(t) \,$ converges to $\, t\,$ and is such that $\, ( \epsi^n(\tau^n(t)),  z^n(\tau^n(t)))\,$ converges to $\, (\epsi(t),z(t))\,$ by definition. Hence, relation \eqref{s} follows since $\, ( \tau^n(t),\epsi^n(\tau^n(t)),  z^n(\tau^n(t))) \in \S$. In particular, we have proved that $\, (\epsi(t),z(t))\,$ solves (see \eqref{return2})
 $$(\epsi(t),z(t))\in\argmin_{(\epsi,z) \in \Rzn \times \Rzd}\big(W_\rho(\epsi,z)   -  \sigma(t):\epsi +D(z - z(t))\big).$$
Moreover, by construction, we have $\, (\epsi(0),z(0))=(\epsi_0,z_0)$.

We are left to prove that indeed $\, (\epsi,z)\,$ fulfills the energy identity \eqref{e}. 
Relation \eqref{poi} can be rewritten as
\begin{gather}
W_\rho( \epsi^n(t), z^n(t)) - \sigma(\tau^n(t)): \epsi^n(t) 
+ \Diss_D(z^n,[0,\tau^n(t)])\nonumber\\
 \leq W_\rho( \epsi_0,z_0) -\sigma(0):\epsi_0 -\int_{0}^{\tau^n(t)} \dot \sigma: \epsi^n\, ds.\label{preupper}
\end{gather}
Hence, passing to the $\, \liminf\, $ in the latter relation and exploiting once again the lower semicontinuity of $\, W_\rho$, the integrability of $\, \dot \sigma$, the boundedness of $\, \epsi^n\,$ (see \eqref{bound2}), and \eqref{propfi}, we readily check by Lebesgue dominated convergence that
\begin{gather}
 W_\rho(\epsi(t),z(t)) - \sigma(t):\epsi(t) + \Diss_D(z,[0,t]) \nonumber\\
\leq W_\rho(\epsi_0,z_0)- \sigma(0):\epsi_0 - \int_0^t \dot \sigma: \epsi \, ds.\label{upper}
\end{gather}

Some more precise convergence for the energy can be deduced.
Indeed, from the stability condition $\, (\epsi^n(t),z^n(t))\in S(s^n(t)) $, the lower semicontinuity of $\, W_\rho$, and the continuity of $\, \sigma \,$ one checks that
\begin{gather}
 W_\rho(\epsi(t),z(t))- \sigma(t): \epsi(t) = \lim_{n\rightarrow +\infty}\big(W_\rho(\epsi(t),z(t))- \sigma(s^n(t)): \epsi(t) + D(z(t) -  z^n(t)) \big)\nonumber\\
 \geq \limsup_{n\rightarrow +\infty}\big( W_\rho(\epsi^n(t),z^n(t)) - \sigma(s^n(t)): \epsi^n(t)\big)\geq W_\rho(\epsi(t),z(t))- \sigma(t): \epsi(t). \label{lim}
\end{gather}
In particular, we have proved that $\, W_\rho(\epsi^n(t),z^n(t))\,$ converges to $\, W_\rho(\epsi(t),z(t))$.

Our next step will be that of proving that $\, (\epsi,z) \,$ is absolutely continuous. Indeed this follows at once from the stability condition \eqref{s}, the upper energy estimate \eqref{upper}, the uniform convexity of $\, W_{\rho}$, and the absolute continuity of $\, \sigma$. Let us fix $\, [s,t] \subset [0,T]$. Owing to $\, (\epsi(s),z(s))\in S(s)$ and the uniform convexity of $\, W_\rho\,$ with constant $\, \alpha\,$ one readily gets that
  \begin{eqnarray}
    &&\alpha |\epsi(t) - \epsi(s)|^2 + \alpha |z(t) - z(s)|^2\nonumber\\ 
&&\leq W_\rho(\epsi(t), z(t)) - \sigma(s):\epsi(t) +D(z(t) - z(s))
 - W_\rho(\epsi(s),z(s)) + \sigma(s):\epsi(s)\nonumber\\
&&\leq W_\rho(\epsi(t), z(t)) - \sigma(t):\epsi(t) + \Diss_D(z,[s,t])\nonumber\\
&&  - W_\rho(\epsi(s),z(s)) + \sigma(s):\epsi(s)- (\sigma(s) - \sigma(t)):\epsi(t)\nonumber\\
&&\leq-\int_s^t\dot\sigma(r):(\epsi(r)- \epsi(t)) \, dr.\nonumber
  \end{eqnarray}
Hence, by means of Gronwall's lemma, one checks that
\begin{equation}\label{lippa}
 |\epsi(t) - \epsi(s)| +|z(t) - z(s)| \leq c_4 \int_s^t|\dot \sigma|,
\end{equation}
where the positive constant $\, c_4 \,$ depends just on $\, \alpha$. The absolute continuity of $\, \epsi\,$ and $\, z \,$ follows.

We are now in the position of proving the converse inequality to \eqref{upper}, namely, the lower energy estimate. Indeed, for all $\, t \in [0,T]$,
\begin{gather}
 W_\rho(\epsi(t),z(t)) - \sigma(t) : \epsi(t) + \Diss_D(z,[0,t]) \nonumber\\
\geq  W_\rho(\epsi_0,z_0)  - \sigma(0):\epsi_0 - \int_0^t  \dot \sigma :\epsi \, ds.\label{lower}
\end{gather}
Indeed, let suitable partitions $\, Q^m=\{0=s^m_0<s^m_1< \dots< s^m_{M^m-1}<s^m_{M^m}=t\}\,$ be given such that the diameters $\, \max_{j=1,\dots,M^m}(s^m_{j}- s^m_{j-1})\,$ go to zero. By exploiting again the stability  $\,  (\epsi(s^m_{j-1}),z(s^m_{j-1})) \in S(s^m_{j-1}) \,$ for $\, j=1, \dots, M^m$, we obtain that
\begin{gather}
 W_\rho(\epsi(s^m_j),z(s^m_j))  - \sigma(s^m_j):\epsi(s^m_j)+ D(z(s^m_j) - z(s^m_{j-1}))\nonumber\\
 \geq W_\rho(\epsi(s^m_{j-1}),z(s^m_{j-1}))  - \sigma(s^m_{j-1}):\epsi(s^m_{j-1})- (\sigma(s^m_{j}) -\sigma(s^m_{j-1})) :\epsi(s^m_j)\nonumber
\end{gather}
We shall take the sum above for $\, j=1, \dots,M^m \,$ and obtain that 
\begin{gather}
 W_\rho(\epsi(t),z(t)) - \sigma(t) : \epsi(t) + \Diss_D(z,[0,t]) \nonumber\\
\geq  W_\rho(\epsi_0,z_0)  - \sigma(0):\epsi_0 - \sum_{j=1}^{M^m} (\sigma(s^m_{j}) - \sigma(s^m_{j-1})) :\epsi(s_j^m).\label{lower2}
\end{gather}
Then, relation \eqref{lower} follows at once from Lebesgue dominated convergence since
$$-\sum_{j=1}^{M^m} (\sigma(s^m_{j}) - \sigma(s^m_{j-1})) :\epsi(s_j^m) = -\int_0^t \left(-\!\!\!\! \!\!\int_{Q^m}\dot \sigma \,dr\right)(s): \epsi(\tau^m(s)) \,ds,$$
where we used a standard notation for the piecewise mean on the partition $\, Q^m$.
In fact, $\, \epsi \circ \tau^m \,$  and $\, -\!\!\!\! \!\!\int_{Q^m}\dot \sigma \,dr\,$ converge to $\, \epsi \,$ and $\,\dot \sigma \,$ at least almost everywhere, respectively, and $\, \epsi\circ \tau^m \,$ is uniformly bounded. Once \eqref{lower} is established, it is a standard matter to check that indeed $\, \Diss_D(z,[0,t])= \phi(t) \,$ for all $\, t \in [0,T]$.
\end{proof}

Finally, \rosso a consequence of \eqref{lippa} is \nero the following Lipschitz regularity result.

\begin{corollary}[Lipschitz continuity]
  Under the assumptions of Theorem \emph{\ref{const}}, if \linebreak$ \sigma \in W^{1,\infty}(0,T;\Rzn)$, then we have $\, (\epsi,z) \in W^{1,\infty}(0,T;\Rzn\times \Rzd)$.
\end{corollary}

We shall complement the above detailed existence analysis by providing a local Lipschitz continuous dependence result for the smooth case $\, \rho >0\,$ (see \cite[Thm. 7.4]{Mielke-Theil04}).

\begin{theorem}[Continuous dependence for $\, \boldsymbol \rho \boldsymbol >\boldsymbol 0$]\label{con_dep_lemma}
Let the assumptions of Theorem~\emph{\ref{const}} hold $\, \rho >0$, $\, \sigma_1, \, \sigma_2 \in W^{1,1}(0,T;\Rzn)$, suitably stable initial data $\, (\epsi_{0,1},z_{0,1})\, $ and $\, (\epsi_{0,2},z_{0,2})\,$ be given and $\, (\epsi_1, z_1 ) \,$ and $\, (\epsi_2,z_2)\,$ be two corresponding energetic solutions to \eqref{s}-\eqref{e}. Then, there exists a positive constant $\, c \,$ depending only on $\, \alpha $,\rosso\
 the bound and the Lipschitz constant of $\, \nabla^2 W_\rho$,
\nero and $\, \| \sigma_i\|_{W^{1,1}(0,T;\Rzn)}\,$ for $\, i=1,2\,$ such that 
\begin{gather}
  | (\epsi_1 - \epsi_2)(t)|^2 + | (z_1- z_2)(t)|^2  \nonumber\\
\leq c\left(| \epsi_{0,1} - \epsi_{0,2}|^2 + | z_{0,1}- z_{0,2}|^2 + \|\sigma_1 - \sigma_2\|^2_{W^{1,1}(0,t;\Rzn)}\right) \quad \forall t \in [0,T]. \label{cont_dep1}
\end{gather}
\end{theorem}

\begin{proof}
  Let us start by introducing some convenient notation. In particular, let
\begin{gather}
 y_i:= \binom{\epsi_i}{z_i}, \quad \nabla W_i := \binom{\partial_\epsi W_\rho(\epsi_i,z_i)}{\partial_z W_\rho(\epsi_i,z_i)},\nonumber\\
 \nabla^2W_i := \binom{\partial_{\epsi\epsi} W_\rho(\epsi_i,z_i)\quad\partial_{\epsi z} W_\rho(\epsi_i,z_i)}{\partial_{\epsi z} W_\rho(\epsi_i,z_i)\quad\partial_{z z} W_\rho(\epsi_i,z_i)} \quad \text{for} \ \ i=1,2.\nonumber
\end{gather}
Next, by exploiting the above mentioned equivalence between \eqref{s}-\eqref{e} and \eqref{evol2}, one readily checks that
\begin{equation}
  \label{serve}
  (\nabla W_1 - \nabla W_2) \cdot (\dot y_1 - \dot y_2) \leq (\sigma_1 - \sigma_2):(\dot \epsi_1 - \dot \epsi_2) \quad \text{a.e. in} \ \ (0,T),
\end{equation}
where of course $\, \cdot \,$ is the scalar product in $\, \Rzn \times \Rzd$. Moreover, we shall use $\,  \overline \epsi := \epsi_1 - \epsi_2$, $\,  \overline z:= z_1 - z_2\,$ and so on. Within this proof, the symbol $\, c \,$ will denote any positive constant possibly depending on $\, \alpha ,\, \| W_\rho\|_{C^{2,1}(\Rzn\times \Rzd)}$, and on $\, \| \sigma_i\|_{W^{1,1}(0,T;\Rzn)}\,$ for $\, i=1,2$. Let us define 
\begin{gather}
 \gamma:= \overline{\partial_\epsi W}:\overline \epsi + \overline{\partial_z W}:\overline z
\geq \alpha|\overline \epsi|^2 + \alpha|\overline z|^2 = \alpha |\overline y|^2,\nonumber
\end{gather}
where we also used the uniform convexity of $\, W_\rho$. Now, by differentiating $\, \gamma \,$ with respect to time and exploiting the smoothness of $\, W_\rho $, one gets that
\begin{eqnarray}
  \dot \gamma &=& (\nabla W_1 - \nabla W_2 + \nabla^2W_1\overline y)\cdot \dot y_1 - (\nabla W_1 - \nabla W_2 + \nabla^2W_2\overline y)\cdot \dot y_2\nonumber\\
&\leq& 2 (\nabla W_1 - \nabla W_2) \cdot (\dot y_1 - \dot y_2) \nonumber\\
&+& |-\nabla W_1 + \nabla W_2 + \nabla^2W_1\overline y| \, |\dot y_1|+ |-\nabla W_2 + \nabla W_1 - \nabla^2W_2\overline y| \, |\dot y_2|\nonumber\\
&\leq&  2 {\overline \sigma}: \dot {\overline \epsi} + c (|\dot y_1|+|\dot y_2|) |\overline y|^2\quad \text{a.e. in} \ \ (0,T).\nonumber
\end{eqnarray}
By collecting the above computation we check that, for all $\, t \in [0,T]$,
\begin{gather}
\gamma(t) = \gamma(0) +\int_0^t \gamma\, ds\leq \gamma(0) + 2\overline \sigma(t) : \overline \epsi(t) - 2\overline \sigma(0) : \overline \epsi_0 -2 \int_0^t \dot {\overline \sigma}: \overline \epsi\, ds+ c \int_0^t (|\dot y_1| +|\dot y_2|)\, \gamma\, ds\nonumber\\
\leq \frac12 \gamma(t) + c\left(| \overline \epsi_{0}|^2 + | \overline z_{0}|^2 + |\overline \sigma(t)|^2+ |\sigma(0)|^2 + \int_0^t (|\dot y_1| +|\dot y_2|)\gamma\, ds\right).\nonumber
\end{gather}
 The assertion follows by Gronwall's lemma.
\end{proof}

\paragraph{Properties of the approximations.} The above detailed existence proof exploits a discrete construction which is interesting in itself. Let us condense in the following lemma the above proved results on the discrete scheme. Note that the result is less sharp for $\, \rho=0\,$ since we do not know whether the solutions are unique in this case.

\begin{lemma}[Convergence]\label{dis}
  Under the assumptions of Theorem \emph{\ref{const}}, the incremental solutions $\, ( \epsi^n,  z^n)\,$ of problem \eqref{return} for partitions $\, P^n\,$ with diameters $\, \tau^n\,$ going to zero are such that, possibly extracting a not relabeled subsequence, for all $\, t \in [0,T]$,
\begin{eqnarray}
  &&  z^n \rightarrow z\quad \text{uniformly in}  \ \ [0,T] ,\nonumber\\
&& \Diss_D(z^n,[0,t]) \rightarrow \Diss_D(z,[0,t]), \nonumber\\
 && \epsi^n(t) \rightarrow \epsi(t),\nonumber\\
&&W_\rho(\epsi^n(t), z^n(t)) \rightarrow W_\rho( \epsi(t), z(t)),\nonumber
\end{eqnarray}
for some pair $\, (\epsi,z) \,$ which solves \eqref{s}-\eqref{e}. As $\, \rho >0 \,$ the whole sequence $\, ( \epsi^n,  z^n)\,$  converges.
\end{lemma}

We conclude this section by recalling from \cite{Mielke-Theil04} (see also \cite{Mielke05}) an a priori error estimate of order $\, 1/2\,$ for the above discussed discrete approximations. The latter error bound is however restricted the smooth situation $\, \rho > 0$.

\begin{lemma}[Error]
Under the assumptions of Lemma \emph{\ref{dis}}, let $\, \rho >0$. Then there exists a positive constant $\, c \,$ depending on $\, \alpha$,\rosso\
the bound and the Lipschitz constant of $\, \nabla^2 W_\rho$,
\nero $\, \| W_\rho\|_{C^{2,1}(\Rzn\times \Rzd)}$, $\, (\epsi_0,z_0)$, and $ \,\| \sigma\|_{W^{1,1}(0,T;\Rzn)}\,$ such that
\begin{gather}
  |(\epsi - \epsi^n)(t)| + |(z - z^n)(t)| \leq c(\tau^n)^{1/2} \quad \forall t \in [0,T].
\end{gather}
\end{lemma}

We shall not provide here a proof of the above lemma. Indeed, in case $\, \sigma \in W^{1,\infty}(0,T;\Rzn)\,$ it suffices to rewrite in the current setting the argument of \cite[Thm. 4.3]{Mielke05}. Moreover, the proof can be adapted with little additional intricacy for the current absolutely continuous case $\, \sigma \in W^{1,1}(0,T;\Rzn)\,$ as well.

\section{Incremental minimization for the boundary\\ value problem}\label{incremental}
\setcounter{equation}{0}

In this section we focus on a minimum problem which arises from the time incremental approximation of the quasi-static evolution. Since we are actually dealing with a rate-independent evolution, this minimum problem is of course the basic tool for understanding the phenomenon. Moreover, the study of the time discrete seems to be heavily addressed by the engineering community \cite{Hackl03,Miehe03,Miehe02,Miehe99,Ortiz-Repetto99,Ortiz00,Ortiz99}. Finally, the time incremental situation will turn out to be better suited than the time-continuous one in order to prove convergence of space approximations.
 
The data of the minimum problem are the current value $\, \overline z \in L^2(\Omega, \Rzd)\,$ of the inelastic strain and the updated values $\, \uD \in H^1(\Omega; \Rz^3)\,$ of the boundary displacement and $\, \ell \in ( H^1(\Omega; \Rz^3))' \,$ of the total load. We shall be interested in solving the following 
\begin{gather}
(u,z) \in \argmin_{(v,w) \in \Y(\uD)}\big(\W_{\rho,\nu}(v,w)-\lan \ell, v\ran  + \D(w- \overline z) \big).\label{min}
\end{gather}
 
The existence of minimizers to the latter problem is a straightforward application of the Direct Method of the Calculus of Variations \cite{Dacorogna89}. Indeed, $\, (v,w) \mapsto \W_{\rho,\nu}(v,w) + \D(w -\overline z) -\lan \ell, v\ran \,$ is trivially coercive and lower semicontinuous  with respect to the weak topology in $\, \Y \,$ and $\, \Y( \uD) \,$ is convex and closed. As far as uniqueness is concerned one should observe that $\, \W_{\rho,\nu} \,$ is uniformly convex for all $\, \rho,\, \nu \geq 0$.

Let us state here a preliminary lemma whose proof can be obtained by means of standard computations on the quadratic form $\, \CC$.

\begin{lemma}[Change of boundary conditions]\label{struc} Let $\, \uD, \,  v^{\text{\rm Dir}}  \in H^1(\Omega; \Rz^3)$, $\, \overline z \in L^2(\Omega, \Rzd)$, and $\, \ell \in ( H^1(\Omega; \Rz^3))' \,$ be given. Moreover, let $\, (u^*,z^*)\in \Y( \uD )\,$ solve \eqref{min} and $\, v^*= u^* - \uD +  v^{\text{\rm Dir}} $. Then $\, (v^*,z^*)\,$ solves
  \begin{equation}
    \label{min2}
    (v^*,z^*)\in \argmin_{(v,z) \in \Y( v^{\text{\rm Dir}})}\left(\W_{\rho,\nu}(v,z) + \int_\Omega \C(\epsi(v) - z): \epsi(\uD - v^{\text{\rm Dir}}) -\lan \ell, v\ran + \D(z- \overline z)\right).
  \end{equation}
On the other hand let $\, (v^*,z^*) \,$ solve \eqref{min2}. Then $\, (v^*-v^{\text{\rm Dir}}+ \uD,z^*)\,$ solves \eqref{min}.
\end{lemma}

 Problem \eqref{min} is H\"older continuously stable with respect to perturbations on the data $\, \overline z, \, \uD$, and $\, \ell$. Indeed, we have the following generalization of Lemma \ref{cont_dep_const}.
\begin{lemma}[Continuous dependence]\label{continuous_dependence}
   Let $\, \rho, \, \nu \geq 0\,$ be fixed and $\, \overline z_1,\, \overline z_2 \in L^2(\Omega, \Rzd)$, $\, \uD_1, \, \uD_2 \in H^1(\Omega; \Rz^3)$, and $\, \ell_1, \, \ell_2 \in ( H^1(\Omega; \Rz^3))'$ be given. Moreover, let $\, (u_i,z_i)\in \Y(\uD_i )\,$ solve \eqref{min} with $\, \uD = \uD_i$, $\, \overline z = \overline z_i$, and $\, \ell = \ell_i\,$ for $\, i=1,2$. Then, there exists a constant $\, c \,$ depending on $\, c_0, \alpha$, $\, C_D$,  and $\, \C \, $ such that 
   \begin{gather}\label{cont_dep}
\| u_1 - u_2\|^2_{H^1(\Omega;\Rz^3)} + \| z_1 - z_2 \|^2_{L^2(\Omega;\Rzd)} + \nu\| z_1 - z_2 \|^2_{H^1(\Omega;\Rzd)} \nonumber\\
\leq c \left(\| \uD_1 - \uD_2\|^2_{H^1(\Omega;\Rz^3)} + \|\overline z_1 -\overline z_2 \|_{L^1(\Omega;\Rzd)}+ \|\ell_1 - \ell_2 \|^2_{(H^1(\Omega;\Rz^3))'}\right).
   \end{gather}
\end{lemma}

\begin{proof} We simply adapt the argument of Lemma \ref{cont_dep_const}
  Owing to the minimality of $\, (u_1,z_1)\,$ and the uniform convexity of $\, \W_{\rho,\nu}\,$ we readily deduce that, for any $\,  (v_1,w_1) \in \Y(\uD_1)$,
\begin{gather}
\alpha\|\epsi(u_1 - v_1) \|_{L^2(\Omega;\Rzn)}^2 + \alpha \| z_1 - w_1 \|^2_{L^2(\Omega;\Rzd)} + \alpha \nu  \| z_1 - w_1 \|^2_{H^1(\Omega;\Rzd)}\nonumber\\
\leq \W_{\rho,\nu}(v_1,w_1) - \lan \ell_1,v_1  \ran + \D(w_1 - \overline z_1)\nonumber\\
- \W_{\rho,\nu}(u_1,z_1) + \lan \ell_1,u_1 \ran- \D(z_1 - \overline z_1) .\nonumber
\end{gather}
On the other hand, the minimality of $\, (u_2,z_2)\,$ entails that, for all $\, (v_2,w_2) \in \Y(\uD_2)$,
\begin{gather}0\leq  \W_{\rho,\nu}(v_2,w_2)  - \lan \ell_2,v_2 \ran+ \D(w_2 - \overline z_2)- \W_{\rho,\nu}(u_2,z_2)+ \lan \ell_2,u_2 \ran - \D(z_2 - \overline z_2).\nonumber
\end{gather}
By choosing $\, (v_1,w_1)=(u_2 - \uD_2+ \uD_1,z_2) \,$ and $\, (v_2,w_2)=(u_1 - \uD_1 + \uD_2,z_1)\,$ and taking the sum of the corresponding inequalities one easily deduces that 
\begin{gather}
\alpha\|\epsi(u_1 - u_2) -\epsi(\uD_1 - \uD_2) \|_{L^2(\Omega;\Rzn)}^2 + \alpha\| z_1 - z_2 \|^2_{L^2(\Omega;\Rzd)}+ \alpha \nu  \| z_1 - z_2 \|^2_{H^1(\Omega;\Rzd)} \nonumber\\
\leq 2\CC(\epsi(\uD_1 - \uD_2)) - \int_\Omega \C\big(\epsi(u_1 - u_2) - (z_1 - z_2)\big):\epsi(\uD_1 - \uD_2) \nonumber\\
+  2 \D(\overline z_1 - \overline z_2) + \lan\ell_1 - \ell_2 ,u_1 - u_2 \ran - \lan \ell_1 - \ell_2, \uD_1 - \uD_2 \ran.\nonumber
\end{gather}
Hence, we readily find a positive constant $\, c \,$ depending on $\, \alpha$, $\, C_D $, and $\, \C \,$ in such a way that
\begin{gather}
\|\epsi(u_1 - u_2) \|_{L^2(\Omega;\Rzn)}^2 + \| z_1 - z_2 \|^2_{L^2(\Omega;\Rzd)}+  \nu  \| z_1 - z_2 \|^2_{H^1(\Omega;\Rzd)} \nonumber\\
\leq c\left( \|\uD_1 - \uD_2\|^2_{H^1(\Omega;\Rz^3)} + \|\overline z_1 - \overline z_2\|_{L^1(\Omega;\Rzd)} + \|\ell_1 - \ell_2 \|^2_{(H^1(\Omega;\Rz^3))'}\right).\nonumber
\end{gather}
Whence, the assertion follows from Korn's inequality \eqref{korn}.
\end{proof}

\paragraph{Convergence of space approximations.} Let us now turn our attention to some space approximation procedure and recall the material of Section \ref{math}. We denote by $\, \Y_{h,0}\,$ the set $\,\Y_{h,0}:= \Y_h \cap \Y(0)$. Given $\, (\tilde u, \tilde z)= p_h^\nu(u,z) \,$ we shall also denote by $\, p^\nu_{h,1}(u,z) := \tilde u \,$ and $\, p^\nu_{h,2}(u,z) := \tilde z$. For the sake of completeness, we shall consider also some approximate situation. Indeed, we ask that for each $\, (\uD,\ove z) \in  \Y \,$ and $\, \ell \in  (H^1(\Omega; \Rz^3))')$, there exist $\, ( \uD_h,\overline z_h) \in \Y_h \,$ {and} $\, \ell_h \in  (H^1(\Omega; \Rz^3))' \,$ such that 
\begin{gather}
(\uD_h,\overline z_h) \rightarrow (\uD,\overline z) \quad \text{strongly in} \ \ H^1(\Omega;\Rz^3)\times L^1(\Omega;\Rzd),\nonumber\\
 \text{and} \ \ \ell_h \rightarrow \ell \quad \text{strongly in} \ \ (H^1(\Omega;\Rz^3))'. \label{err}
\end{gather}
 We shall be concerned with the approximating minimum problem
\begin{gather}
  (u_h,z_h) \in \argmin_{(u-  \uD_h,z) \in  \Y_{h,0}}\big(\W_{\rho,\nu}(u,z) -\lan \ell_h, u\ran  + \D(z- \overline z_h)\big).\label{min3}
\end{gather}
The latter problem is of course uniquely solvable since $\, (u,z) \mapsto \W_{\rho,\nu}(u,z) -\lan \ell_h, u\ran + \D(z -\overline z) \,$ is again uniformly convex, coercive, and lower semicontinuous in $\, \Y_h \,$ and $\, \Y_{h,0}\,$ is convex and closed.

Assuming \eqref{err} and letting $\, (u,z) \,$ and $\, (u_h,z_h) \,$ solve the minimum problem \eqref{min} and \eqref{min3}, respectively, the main issue of this section is that of proving that $\, (u_h,z_h) \,$ converges to $\, (u,z) \,$ strongly in $\, \Y$. More precisely, in the case $\, \rho > 0$, some quantitative error estimates can be obtained.

\begin{lemma}[Error for $\, \boldsymbol \rho\boldsymbol >\boldsymbol 0$]\label{error}
 Let $\, \rho>0, \,\nu\geq 0 \,$ be given and $\, (u,z) \,$ and $\, (u_h,z_h) \,$ solve \eqref{min} and \eqref{min3}, respectively. Moreover, let 
\begin{equation}\label{espla} 
\lan \ell_h, v - p^\nu_{h,1}(v,w)\ran =0\quad \text{for all}\quad (v,w) \in \Y \ \ \text{and} \ \  h>0.
\end{equation}
Then, there exists a positive constant $\, c \,$ depending on $\,\rho,\, c_0, \, \alpha,\, C_D$, and $\, \C\,$  such that
  \begin{gather}
 \| u - u_h\|^2_{H^1(\Omega;\Rz^3)} + \| z - z_h \|^2_{L^2(\Omega;\Rzd)} + \nu\| z - z_h \|^2_{H^1(\Omega;\Rzd)} \nonumber\\
\leq   c\left(\| \uD - \uD_h\|_{H^1(\Omega;\Rz^3)}^2 + \| \overline z - \overline z_h\|_{L^1(\Omega;\Rzd)} \right)\nonumber\\
+ c \left( \|\ell - \ell_h \|^2_{(H^1(\Omega; \Rz^3))'} +\|  z -  p^\nu_{h,2}(v,z)\|_{L^1(\Omega;\Rzd)}\right).\label{err2}
  \end{gather}
\end{lemma}
Let us comment that \eqref{espla} turns out to be fulfilled in the frame of conforming finite elements. Considering for simplicity the case where $\, p_{h,1}^\nu\,$ does not depend on $\, w$, a fairly usual choice for $\, \ell_h \,$ is 
$$\lan \ell_h, v \ran := \lan \ell, p_{h,1}^\nu(v)\ran\quad \forall v \in \U,$$
whence \eqref{espla} follows.
\begin{proof}
  The estimate follows by carefully reconsidering the continuous dependence proof of Lemma \ref{continuous_dependence} and exploiting Galerkin's orthogonality \eqref{gal}. Indeed, making use of Lemma \ref{struc}, one obtains for $\, v= u -\uD \,$ and $\, v_h = u_h - \uD_h$,
  \begin{gather}
    \alpha\|\epsi(v - v_h) \|_{L^2(\Omega;\Rzn)}^2 + \alpha \| z - z_h \|^2_{L^2(\Omega;\Rzd)} + \alpha \nu  \| z - z_h \|^2_{H^1(\Omega;\Rzd)}\nonumber\\
\leq \A(v_h,z_h) + \G_{\rho}(z_h) + \int_\Omega \C(\epsi(v_h) - z_h): \epsi(\uD) + \D(z_h - \overline z) - \lan \ell, v_h - v\ran \nonumber\\
 -\A(v,z) - \G_{\rho}(z) - \int_\Omega \C(\epsi(v) - z): \epsi(\uD) - \D(z - \overline z)\label{bla0}
  \end{gather}
where we have denoted by $\, \G_{\rho}: L^2(\Omega, \Rzd)\rightarrow [0,+\infty]\,$ the convex functional
$$\G_\rho(z) := \F_\rho(z) - c_2 \|z \|^2_{L^2(\Omega, \Rzd)}.$$
Moreover, arguing exactly as in Lemma \ref{continuous_dependence} and defining $\,(\tilde v,\tilde z):= p^\nu_{h}(v,z)$, we readily check that
\begin{gather}
  0\leq  \A(\tilde v,\tilde z) + \G_{\rho}(\tilde z) + \int_\Omega \C(\epsi(\tilde v) - \tilde z): \epsi(\uD_h) + \D(\tilde z - \overline z_h) - \lan \ell_h, \tilde v - v_h\ran \nonumber\\
 -\A(v_h,z_h) - \G_{\rho}(z_h) - \int_\Omega \C(\epsi(v_h) - z_h): \epsi(\uD_h) - \D(z_h - \overline z_h). \label{bla}
\end{gather}
Taking the sum of the latter inequalities and exploiting \eqref{gal2}, \eqref{espla}, and $\, (\uD_h,0) \in \Y_h$, we easily check that 
  \begin{gather}
    \alpha\|\epsi(v - v_h) \|_{L^2(\Omega;\Rzn)}^2 + \alpha \| z - z_h \|^2_{L^2(\Omega;\Rzd)} + \alpha \nu  \| z - z_h \|^2_{H^1(\Omega;\Rzd)}\nonumber\\
\leq \int_\Omega \C(\epsi(v_h -v) - (z_h-z)): \epsi(\uD- \uD_h) + 2 \D(\overline z - \overline z_h) \nonumber\\
+ \lan \ell - \ell_h , v - v_h \ran + \G_\rho(\tilde z) - \G_\rho(z) + \D(z - \tilde z),\nonumber
  \end{gather}
and the assertion follows.
\end{proof}

We shall now turn to some (necessarily weaker) quantitative convergence estimate for the specific case $\, \rho =0$.

\begin{lemma}[Convergence for $\, \boldsymbol \rho\boldsymbol =\boldsymbol 0$]\label{error2}
 Under the assumptions of Lemma \emph{\ref{error}}, let $\, \rho = 0$.
Moreover, let $\, (\tilde v, \tilde z):= p_h^\nu(u-\uD,z)\,$ and $\, (\hat v, \hat z):= (q_h(u-\uD),r^\nu_h(w))$. Then, there exists a positive constant $\, c \,$ depending on $\,c_1 \,$ and the same constant of \eqref{err2} such that
  \begin{gather}
 \| u - u_h\|^2_{H^1(\Omega;\Rz^3)} + \| z - z_h \|^2_{L^2(\Omega;\Rzd)} + \nu\| z - z_h \|^2_{H^1(\Omega;\Rzd)} \nonumber\\
\leq   c\left(\| \uD - \uD_h\|_{H^1(\Omega;\Rzn)}^2 + \| \overline z - \overline z_h\|_{L^1(\Omega;\Rzd)} + \|\ell - \ell_h \|^2_{(H^1(\Omega; \Rz^3))'} +\|  z -  \tilde z\|_{L^1(\Omega;\Rzd)}\right)\nonumber\\
+c\left( \A(\hat v, \hat z) - \A(\tilde v, \tilde z) + \int_\Omega \C(\epsi(\hat v - \tilde v) - (\hat z - \tilde z)):\epsi (\uD_h)\right)\nonumber\\
 + c\left(\lan \ell_h, \tilde v - \hat v\ran + \|\hat z - \tilde z \|_{L^1(\Omega;\Rzd)}\right) .\label{err3}
  \end{gather}
\end{lemma}

Since of course $\, p_h(v,w)- (q_h(v),r_h^\nu(w))\,$ strongly converges to zero in $\, \Y$, estimate \eqref{err3} proves in particular that, assuming \eqref{err}, the strong convergence of the approximations holds.

\begin{proof} This proof follows the same lines of Lemma \ref{error}.
We shall however replace \eqref{bla} as follows.
 \begin{gather}
  0\leq  \A(\hat v,\hat z) + c_1\|\hat z\|_{L^1(\Omega;\Rzd)} + \int_\Omega \C(\epsi(\hat v) - \hat z): \epsi(\uD_h) + \D(\hat z - \overline z_h) - \lan \ell_h, \hat v - v_h\ran \nonumber\\
 -\A(v_h,z_h) - c_1\|z_h\|_{L^1(\Omega;\Rzd)} - \int_\Omega \C(\epsi(v_h) - z_h): \epsi(\uD_h) - \D(z_h - \overline z_h), \nonumber
\end{gather} 
and again take its sum with \eqref{bla0}. In order to reduce to the situation of Lemma \ref{error} one needs to simply add and subtract the term $\,\tilde z\,$ in most of the occurrences of $\, \hat z$. This procedure of course produces the extra residual terms that appear in the last two lines of \eqref{err3}.
\end{proof}


\section{The incremental problem.}\label{INCRE} 

We shall prepare here some material in the direction of the full time-stepping procedure. To this aim, we assume to be given a partition $\,  P:=\{0=t_0<t_1<\dots<t_{N-1}<t_{N}=T\} \,$ with diameter $\, \tau=\max_{i=1,\dots,N}(t_i - t_{i-1}) \,$ and data $\, \{\uD_i \}_{i=0}^N \in (H^1(\Omega;\Rz^3))^{N+1}$, $\,  \{ \ell_i \}_{i=0}^N \in ((H^1(\Omega;\Rz^3))')^{N+1}$, and $\, (u_0,z_0) \in \Y(\uD_0)$. Hence, we find iteratively the unique solutions $\, \{(u_i,z_i)\}_{i=1}^N \,$ to the problem
\begin{gather}
  (u_i,z_i) \in \argmin_{(u,z)\in \Y(\uD_i)}\big(\W_{\rho,\nu} (u,z) - \lan \ell_i, u\ran + \D(z- z_{i-1})\big)\quad \text{for} \ \ i=1, \dots,N.\label{pass}
\end{gather}
We shall denote by  $\, (u , z)\,$ the incremental solution which interpolates right-continuously the values $\, (u_i,z_i)\,$ on the partition $\, P$. Hence, the following a priori estimate holds true.

\begin{lemma}[A priori bounds]\label{apri}
 Let $\, \rho, \, \nu \geq 0$. Then there exists a positive constant $\, c \,$ depending on $\, \alpha, \, \W_{\rho,\nu}(u_0,z_0),\, \lan \ell_0,u_0\ran$, and $\, \sum_{i=1}^{N}\| \ell_i - \ell_{i-1}\|_{(H^1(\Omega;\Rz^3))'}\,$ such that 
  \begin{equation}
    \label{a_priori}
    \W_{\rho,\nu}(u,z) + \Diss_\D(z,[0,T])\leq c.
  \end{equation}
\end{lemma}
\begin{proof}
  From the minimality of $\, (u_i,z_i)\,$ in \eqref{pass} one has that
  \begin{gather}
    \W_{\rho,\nu}(u_i,z_i) - \lan \ell_i , u_i \ran + \D(z_i - z_{i-1})\nonumber\\
 \leq  \W_{\rho,\nu}(u_{i-1}, z_{i-1}) - \lan \ell_{i-1}, u_{i-1}\ran - \lan \ell_i - \ell_{i-1},u_{i-1}\ran.\nonumber
  \end{gather}
Taking the sum in the latter relation for $\, i=1, \dots,m, \ m \leq N$, one has that
  \begin{gather}
    \W_{\rho,\nu}(u_m,z_m) - \lan \ell_m , u_m \ran + \sum_{i=1}^m\D(z_i - z_{i-1})\nonumber\\
 \leq  \W_{\rho,\nu}(u_{0}, z_{0}) - \lan \ell_{0}, u_{0}\ran - \sum_{i=1}^m\lan \ell_i - \ell_{i-1},u_{i-1}\ran.\nonumber
  \end{gather}
and the assertion follows from the uniform convexity of $\, \W_{\rho,\nu}\,$ and the Gronwall lemma.
\end{proof}

Let us collect here some remark on the incremental problem \eqref{pass} in the space discretized situation. To this aim we shall refer to the notation introduced in Section \ref{math} and assume to be given, for all $\, h >0$, suitable data $\, \{\uD_{i,h} \}_{i=0}^N \in ({\cal U}_h)^{N+1}$, $\,  \{ \ell_{i,h} \}_{i=0}^N \in ((H^1(\Omega;\Rz^3))')^{N+1}$, and $\, (u_{0,h},z_{0,h})\,$ such that $\,  (u_{0,h}- \uD_{0,h},z_{0,h}) \in \Y_{0,h}$. Hence, by solving iteratively the minimum problem, we define the right-continuous piecewise constant incremental solutions $\, (u_{h},z_{h})$.

First of all, one should notice that the a priori bound of Lemma \ref{apri} holds for  $\, (u_{h},z_{h})\,$ as well (of course the dependences of the constant are referred to the approximating data). Secondly, we are in the position of obtaining for  $\, (u_{h},z_{h})\,$ the same continuous dependence as in Lemma \ref{continuous_dependence}. This fact entails the convergence of the space approximated incremental problem in $\, N \,$ steps to the corresponding limit. In particular, employing Lemma \ref{error} or \ref{error2}, respectively, and performing an induction over $\, i=1, \dots, N$, we have the following result.

\begin{lemma}[Convergence for $\,\boldsymbol N\,$ steps as $\, \boldsymbol h \boldsymbol \rightarrow \boldsymbol 0$]\label{conve}
 Under the above assumptions, let the parameters $\, \rho,\nu \geq 0$, $\, N \in \Nz \,$ be fixed and assume that $\, \uD_{i,h}\rightarrow \uD_i\,$ in $\, H^1(\Omega;\Rz^3)$, $\,   \ell_{i,h}\rightarrow \ell_i \,$ in $\,(H^1(\Omega;\Rz^3))'$, and $\, (u_{0,h},z_{0,h}) \rightarrow (u_0,z_0)\,$ in $\, H^1(\Omega;\Rz^3) \times L^1(\Omega;\Rzd)\,$ as $\, h \rightarrow 0$. Then, we have that $\,  u_{i,h}\rightarrow  u_i\,$ in $\, H^1(\Omega;\Rz^3)\,$ as well, for all $\, i =1,\dots,N$.
\end{lemma}

Indeed, we would be in the position of stating a more precise quantitative bound for the error $\, \max_{1\leq i \leq N}\|u_{i} - u_{i,h}\|_{H^1(\Omega;\Rz^3)}\,$ in terms of data. This bound will however deteriorate and eventually explode as $\, N \rightarrow + \infty$.

\section{The evolution problem}\label{evoevo}
\setcounter{equation}{0}

We shall finally turn to the study of the time-continuous problem. In particular, we are interested in {\it energetic solutions} to \eqref{evol}-\eqref{equi} along with the above prescribed boundary displacement and boundary traction conditions. Namely, our solutions will be functions $\,  t\mapsto (u(t),z(t)) \in \Y(\uD(t))\,$ such that $\, t \mapsto \lan \dot \ell(t),u(t) \ran \,$ is integrable and, for all $\, t \in [0,T]$,
 \begin{eqnarray}
&&  (u(t),z(t))\in \Big\{(u,z) \in \Y(\uD(t)) \ \ \text{such that},\ \   \forall (\overline u, \overline z) \in  \Y(\uD(t)),\nonumber\\
&&\qquad\qquad\qquad\qquad\W_{\rho,\nu}(u,z) - \lan \ell(t),u \ran \leq \W_{\rho,\nu}(\overline u, \overline z)  - \lan \ell(t),\overline u \ran + \D(z - \overline z)\Big\},  \label{S}\\
&&\W_{\rho,\nu}(u(t),z(t)) -   \lan \ell(t),u(t) \ran + \Diss_\D(z,[0,t]) \nonumber\\
&&\qquad\qquad = \W_{\rho,\nu}(u(0),z(0)) - \lan \ell(0),u(0) \ran - \int_0^t  \lan \dot\ell(s),u(s)\ran\, ds.\label{E}
\end{eqnarray}

Following the argument of Section \ref{constitutive}, we are in the position of proving the equivalence of the two formulations \eqref{evol}-\eqref{equi} and \eqref{S}-\eqref{E} as soon as the above mentioned boundary condition (plus an extra homogeneous Neumann type condition for $\, z \,$ when $\, \nu >0$) are considered and the solutions are assumed to be at least absolutely continuous. The latter is of course a quite natural regularity requirement and we will readily recover it in our framework. 

The main issue of this section is to fix $\, \nu> 0\,$ and exploit the
analysis of \AAAS\cite{Mielke-et-al02,Mielke-Theil04} \AAAE in order to obtain some existence, uniqueness, and convergence of approximations result. Apart from infinite dimensions, the arguments involved here are quite close to those of Section \ref{constitutive}. Owing to this consideration, we will mainly sketch the proofs of the forthcoming results by heavily referring to the corresponding material in Section \ref{constitutive}.

\paragraph{An equivalent problem.} It is convenient to introduce yet another equivalent formulation of problem \eqref{S}-\eqref{E} by replacing the variable $\, u \,$ by $\, v = u - \uD$. The main advantage of this change of variables is that the energetic formulation for $\, (v,z)\,$ takes values in the fixed phase space $\, \Y_0:=\Y(0)$. Indeed, in the same spirit of Lemma \ref{struc}, one readily 
computes that 
$$\W_{\rho,\nu}(u,z)- \lan \ell , u \ran = \W_{\rho,\nu}(v,z) +\int_\Omega \C(\epsi(v) - z):\epsi(\uD) - \lan \ell , v \ran + \CC(\epsi(\uD))- \lan \ell , \uD \ran.$$
Hence, one checks that $\, (u,z) \,$ is an energetic solution if and only if $\, (v,z):t \mapsto \Y_0 \,$ is such that, for all $\, t \in [0,T]$,
\begin{eqnarray}
&&\!\!\!\!\!\!\!\!\!\!\!\!\!\!\!\!\!\!\!\!\!\!\!\!\!\!\!\!\!\!\!\!  (v(t),z(t))\in \S(t):=\Big\{(v,z) \in \Y_0 \ \ \text{such that}, \ \ \forall (\overline v, \overline z) \in  \Y_0
,\nonumber\\
&&\W_{\rho,\nu}(v,z) - \lan L(t),(v,z) \ran \leq \W_{\rho,\nu}(\overline v, \overline z)  - \lan L(t),(\overline v,\overline z)\ran+\D(z - \overline z) \ran \Big\},  \label{SS}\\
&&\!\!\!\!\!\!\!\!\!\!\!\!\!\!\!\!\!\!\!\!\!\!\!\!\!\!\!\!\!\!\!\!\W_{\rho,\nu}(v(t),z(t)) -  \lan L(t),(v(t),z(t)) \ran + q(t)+ \Diss_\D(z,[0,t]) \nonumber\\
&& = \W_{\rho,\nu}(v(0),z(0)) - \lan L(0),(v(0),z(0)) \ran + q(0) \nonumber \\
&&- \int_0^t  \lan \dot\ell(s),v(s)\ran\, ds- \int_0^t  \lan \dot\ell(s),\uD(s)\ran\, ds,\label{EE}
\end{eqnarray}
where we have denoted by $\, L:[0,T]\rightarrow (\Y_0)' \,$ the functional
$$\lan L(t), (v,z)\ran := - \int_\Omega \C(\epsi(v)- z): \epsi(\uD(t)) + \lan \ell(t),v\ran\quad \forall (v,z) \in \Y_0, \ t \in [0,T].$$
Here $\,\lan \cdot,\cdot \ran\,$ is used for the duality pairing between $\, (\Y_0)' \,$ and $\, \Y_0$, as well. Moreover, the function $\, q:[0,T]\rightarrow \Rz \,$ is defined as
$$q(t):= \CC(\uD(t)) - \lan \ell(t), \uD (t) \ran \quad \forall t \in [0,T].$$
We shall explicitly observe that $\, \uD \in W^{1,1}(0,T;H^1(\Omega;\Rz^3))\,$ and $\, \ell \in W^{1,1}(0,T;(H^1(\Omega,\Rz^3))')\,$ entail that $\, L \in W^{1,1}(0,T;(\Y_0)')\,$ and $\, q \in W^{1,1}(0,T)$.

 From now on, we will focus on problem \eqref{SS}-\eqref{EE} and leave to the reader the straightforward interpretation of the forthcoming results for our original variable $\, u$. Let us start from the following existence result.

\begin{theorem}[Existence for $\, \boldsymbol \nu\boldsymbol >\boldsymbol 0$]\label{const2}
Let $\, \nu>0\,$ and $\, \rho \geq 0$. Given $\, L \in W^{1,1}(0,T;(\Y(0))')$, $\, q \in W^{1,1}(0,T)$,  and  $\, (v_0,z_0)\in \S(0) $, there exists an energetic solution $\, (v,z)\,$ to \eqref{SS}-\eqref{EE} such that $\, (v(0),z(0))=(v_0,z_0)$. Moreover $\, (v,z) \in W^{1,1}(0,T;\Y_0)$.
\end{theorem}

We shall not provide here a full proof of this result. Indeed, it suffices to suitably adapt the machinery of Lemma \ref{const} to the situation of \eqref{SS}-\eqref{EE}. In particular, we argue again by discretizing the problem on a sequence of partitions $\, P^n\,$ with diameter going to zero. The corresponding incremental problems 
\begin{equation}
  \label{min4}
    (v_i,z_i)\in\argmin_{(v,z) \in  \Y_0}\big(\W_{\rho,\nu}(v,z)  -\lan L(t^n_i), u\ran + \D(z- z^n_{i-1}) \big)\quad \text{for} \  \ i=1, \dots, N^n,
\end{equation}
will turn out to be solvable by means of the results of Section \ref{incremental}. Namely, we can introduce some right-continuous and piecewise constant interpolant $\, (v^n,z^n)\,$ of the discrete solution on the partition $\, P^n$. Moreover, we exploit Lemma \ref{apri} which entails that
$$\sup_{t \in [0,T]} \W_{\rho,\nu}( v^n(t),  z^n(t)) \ \ \text{and} \ \  \Var_{[0,T]}( z^n) \ \ \text{are bounded independently of} \ \ n.$$
Indeed, the latter bound depends now on $\, \W_{\rho,\nu}(v_0,z_0)$, $\, \|L\|_{W^{1,1}(0,T;(\Y(0)))')}$, and $\, \|q\|_{W^{1,1}(0,T)}$. 

As for the limit, we will make use of some extended version of Helly's principle \cite[Thm. 3.1]{Mainik-Mielke05} and find a (not relabeled) subsequence of partitions and a non-decreasing function $\, \phi:[0,T]\rightarrow [0,+\infty)\,$ such that
\begin{gather}
    z^n(t) \rightarrow z(t) \ \ \text{weakly in} \ \ H^1(\Omega;\Rzd) \ \ \text{and} \ \ \Diss_\D( z^n,[0,t])\rightarrow \phi(t)\quad \text{for all} \ \ t \in [0,T],\nonumber\\
\Diss_\D( z,[s,t])\leq \phi(t) - \phi(s) \quad \forall [s,t]\subset [0,T].\nonumber
\end{gather}
Indeed, here we have used in a crucial way that $\, \nu >0 $, i.e., the sublevels of $\, \W_{\rho,\nu} \,$ are compact in $\, L^2(\Omega; \Rzn)\times L^2(\Omega;\Rzd)$. Moreover, we have that $\,  v^n(t)= {\cal L} z^n(t)$, ${\cal L }\,$ being linear, and  $\, {\cal L} z^n(t)\to {\cal L} z(t) = v(t)\,$ weakly in $\,H^1(\Omega;\Rz^3) \,$ for all $\, t \in [0,T]$, where $\, (v(t),0) \in \Y_0$. 

The set of stable trajectories $\, \S:= \cup_{t \in[0,T]} (t,\S(t))\,$ is closed with respect to the weak topology of $\, \Y$. Namely, letting $\, (t_k,v_k,z_k)\in \S\,$ with $\, t_k \to t\,$ and $\, (v_k,z_k)\to (v,z)\,$ weakly in $\,\Y_0$, we readily exploit the lower semicontinuity of $\, \W_{\rho,\nu}$, the weak continuity of $\, \D \,$ in $\, H^1(\Omega;\Rzd)$, and the continuity of $\, L \,$ and get that 
\begin{gather}
  \W_{\rho,\nu}(v,z) + \lan L(t), (v,z)\ran \leq \liminf_{k \to +\infty} \big(\W_{\rho,\nu}(v_k,z_k) + \lan L(t_k), (v_k,z_k)\ran\big)\nonumber\\
 \leq \liminf_{k \to +\infty}\big(\W_{\rho,\nu}(\overline v,\overline z) + \lan L(t_k), (\overline v,\overline z)\ran + \D(z_k -\overline z)\big)=\W_{\rho,\nu}(\overline v,\overline z) + \lan L(t), (\overline v,\overline z)\ran + \D(z -\overline z)\nonumber
\end{gather}
for all $\, (\overline v,\overline z) \in \Y_0$. Namely, $\, (t,v,z) \in \S\,$ and the stability condition \eqref{SS} easily follows. Moreover, the initial condition is fulfilled by construction and the uniform convexity of $\, \W_{\rho,\nu} \,$ along with stability entail that the whole sequence $\,  \epsi(v^n(t)) \,$ actually converges to $\, \epsi(v(t))$.

As for to prove that $\, (v,z)\,$ fulfills \eqref{EE} we readily deduce from the above stated convergences and lower semicontinuity arguments (see \eqref{preupper}) that the equivalent of \eqref{upper} holds. Indeed we have that
\begin{eqnarray}
&&\W_{\rho,\nu}(v^n(t), z^n(t)) -  \lan L(\tau^n(t)),(v^n(t),z^n(t))\ran + q(\tau^n(t))
+\Diss_\D(z^n,[0,\tau^n(t)])\nonumber\\
&&\qquad\qquad\qquad\qquad \leq \W_{\rho,\nu}( v_{0},z_{0}) -  \lan L(0),(v _{0},z_{0})\ran + q(0)\nonumber\\
&&\qquad\qquad\qquad\qquad-\int_{0}^{\tau^n(t)} \lan \dot \ell(s), v^n(s)\ran\, ds  -\int_{0}^{\tau^n(t)} \lan \dot \ell(s),\overline u(s)\ran\, ds.\label{preupper2}
\end{eqnarray}
and we simply pass to the $\, \liminf \,$ as $\, n \rightarrow +\infty\,$ in order to get that
\begin{eqnarray}
&&\!\!\!\!\W_{\rho,\nu}(v(t),z(t)) -   \lan L(t),(v(t),z(t)) \ran + q(t) + \Diss_\D(z,[0,t]) \nonumber\\
&&\ \leq \W_{\rho,\nu}(v_0,z_0) - \lan L(0),(v_0,z_0) \ran + q(0) - \int_0^t  \lan \dot\ell(s),v(s)\ran\, ds -\int_{0}^{t} \lan \dot \ell(s),\overline u(s)\ran\, ds.\label{upper3}
\end{eqnarray}
Moreover, again by stability, one has that $\, \W_{\rho,\nu}(v^n(t),z^n(t))\rightarrow \W_{\rho,\nu}(v(t),z(t))\,$ as well (see \eqref{lim}). As a by-product, the above stated weak convergence for $\, (v^n(t),z^n(t))\,$ turns out to be actually strong in $\, \Y$. 

Exactly as in Theorem \ref{const}, the absolute continuity of $\, (v,z)\,$ follows at once from that of $\, L \, $ and $\, q$, relation \eqref{upper3}, the uniform convexity of $\, \W_{\rho,\nu}$, and stability \eqref{SS}. In particular, we are in the position of reproducing the same argument as in \eqref{lower2} and, exploiting once more stability and the continuity of data, obtain the upper energy estimate as well. Namely, one has that $\, \phi(t) = \Diss_\D(z,[0,t]) \,$ for all $\, t \in [0,T]$. The existence proof is hence complete.

Again, energetic solutions corresponding to Lip\-schitz continuous data turn out to be Lip\-schitz continuous as well.

\begin{lemma}[Lipschitz continuity]
  Under the assumptions of Theorem \emph{\ref{const2}}, whenever $\, L \in W^{1,\infty}(0,T;(\Y_0)') \,$  and $\, q \in W^{1,\infty}(0,T)$, we have $\, (\epsi,z) \in W^{1,\infty}(0,T;\Y_0)$.
\end{lemma}

\paragraph{Existence by smoothness.} The above sketched existence proof exploits in a crucial way the compactness of the sublevels of $\, \W_{\rho, \nu}\,$ for $\, \nu >0\,$ in the weak topology of $\, H^1(\Omega;\Rz^3) \times H^1(\Omega; \Rzd) \,$ and works for any $\, \rho> 0$. An alternative approach to existence of solutions of the energetic formulation is however available in the  smooth situation $\, \rho >0\,$ by means of the construction of \cite[Sec. 7]{Mielke-Theil04}, for instance. A possible advantage of this perspective is that of gaining explicit convergence rates. We shall address this issue elsewhere.

In the above mentioned smooth situation $\, \rho >0 \,$ no compactness is
assumed for energy-bounded states but the energy functional $\,
\W_{\rho,\nu}:\Y \rightarrow [0,+\infty)\,$ is required to be $\,
C^{2,1}$. This again forces\linebreak $\, \nu>0$. Namely, given $\, h \in
C^{2,1}(\Rz)\,$ \AAAS with $ 
\, h'' \in L^\infty(\Rz)$, \AAAE one has that the functional \linebreak$\, {\cal H} : L^2(\Omega;\Rzd)\rightarrow \Rz\,$ defined by
$${\cal H}u:= \int_\Omega h(u(x))dx \quad \text{for } u \in L^2(\Omega;\Rzd)$$
is $\, C^{2,1}\,$ if and only if $\, h \,$ is quadratic (and in this case $\,{\cal H} \in C^\infty$). On the other hand, $\, {\cal H} \,$ is $\, C^{2,1} \,$ on $\, H^1(\Omega;\Rzd)$. This fact entails that $\, \W_{\rho,\nu }\,$ is $\, C^{2,1}\,$ on $\, \Y \,$ if and only if $\, \nu>0$.

\paragraph{Continuous dependence.} We are in the position of reproducing the continuous dependence result of Section \ref{constitutive} in the present framework and for $\, \rho,\, \nu>0$. Once again continuous dependence relies on uniform convexity and $\, C^{2,1}\,$ continuity of the energy functional. In particular, the assumption $\, \nu >0$, which of course plays no role in Lemma \ref{con_dep_lemma}, is actually needed here (see above). 
\paragraph{Properties of the approximations.} The time discretization technique described above has of course some interest in itself. Let us collect for convenience some related result in the following.

\begin{lemma}\label{dis2}
 Let $\, \nu>0$. Under the assumptions of Theorem \emph{\ref{const2}}, the incremental solutions $\, ( v^n,  z^n)\,$ of problem \eqref{min4} for partitions $\, P^n\,$ with diameters $\,\tau^n\,$ going to $\,0\,$ are such that, possibly extracting a not relabeled subsequence, for all $\, t \in [0,T]$,
\begin{eqnarray}
  &&  z^n \rightarrow z\quad \text{strongly in}  \ \ C([0,T];H^1(\Omega;\Rzd)) ,\nonumber\\
&& \Diss_\D(z^n,[0,t]) \rightarrow \Diss_\D(z,[0,t]), \nonumber\\
 && v^n(t) \rightarrow v(t)\quad \text{strongly in} \ \  H^1(\Omega;\Rz^3),\nonumber\\
&&\W_{\rho,\nu}(v^n(t), z^n(t)) \rightarrow \W_{\rho,\nu}( v(t), z(t)),\nonumber
\end{eqnarray}
for some $\, (v,z) \,$ which solves \eqref{SS}-\eqref{EE}.  As $\, \rho > 0 \,$ the whole sequence is convergent to the unique energetic solution $\, (v,z)\,$ and there exists a positive constant $\, c \,$ depending on $\, \alpha$, $\,\| \W_{\rho,\nu}\|_{C^{2,1}(\Y_0;\Rz)}$, $\, (v_0,z_0)$, $ \,\| L\|_{W^{1,1}(0,T;(\Y(0))')}$, and $ \,\| q\|_{W^{1,1}(0,T)}\,$ such that
\begin{gather}
  \|(v - v^n)(t)\|_{H^1(\Omega;\Rz^3)} + \|(z - z^n)(t)\|_{H^{\nu}(\Omega;\Rzd)}  \leq c (\tau^n)^{1/2} \quad \forall t \in [0,T].
\end{gather}
\end{lemma}

\paragraph{Full space-time approximations.} We conclude this analysis by commenting on the possibility of performing a full space-time approximation of the problem. To this aim let us refer to the above introduced notations, consider some approximation parameter $\, h>0 $, and reduce the energetic formulation \eqref{SS}-\eqref{EE} to the spaces $\, \Y_{h,0} \,$ exhausting $\, \Y_0$. We shall be considering in particular some discrete values $\, \{(v^n_{h,i},z^n_{h,i})\}_{i=0}^{N^n}\,$ defined inductively from suitable initial data $\, (v_{h,0},z_{h,0}) \in \Y_{h,0}\,$ by letting $\, (v^n_0,z^n_0)=(v_{h,0},z_{h,0})\,$ and solving the following incremental problem
\begin{gather}
  (v^n_{h,i},z^n_{h,i})\in\argmin_{(v,z) \in  \Y_{h,0}}\big(\W_{\rho,\nu}(v,z) -\lan L(t^n_i), u\ran + \D(z- z^n_{h,i-1})\big)\quad \text{for} \ \ i=1, \dots, N^n.\label{min5}
\end{gather}
Again, the unique solvability of the latter problems is ensured by uniform convexity and lower semicontinuity, i.e., it is independent of $\, h$. We will denote as usual by $\,(v^n_h,z^n_h)\,$ the corresponding incremental solutions. 

Our first observation is that, arguing exactly as above, whenever the assumptions of Theorem \ref{const2} are fulfilled and the initial data are bounded in energy independently of $\, h$, the usual bound
\begin{equation}\label{bound99}
\sup_{t \in [0,T]} \W_{\rho,\nu}( v^n_h(t),  z^n_h(t)) \ \ \text{and} \ \ \Diss_\D( z^n_h,[0,T]) \ \ \text{are bounded indep. of $\, n \,$ and $\, h$},
\end{equation}
can be obtained. 

\paragraph{Convergence for the space-discretized problem.} Assume $\, h>0$. Then, we are in the position of reproducing the argument of Theorem \ref{const2} and deduce the existence of a limiting space-approximated energetic solution $\, (v_h,z_h)$. To this aim, the restriction $\, \nu >0 \,$ could even be avoided whenever $\,\Y_h \,$ are chosen to be finite dimensional, for instance. Moreover, the fully discrete solution $\,(v^n_h,z^n_h)\,$ converges to $\, (v_h,z_h)$ in the sense of Lemma \ref{dis2} as $\, n \rightarrow +\infty$. We shall not give a detailed proof of these facts but rather limit ourselves in observing that the energetic formulation \eqref{SS}-\eqref{EE} can be rewritten in $\, \Y_{h,0} \,$ with no intricacy. In particular, estimate \eqref{bound99} is again the starting point for the limit procedure.

Once the energetic solution $\, (v_h,z_h): [0,T] \rightarrow \Y_{h,0}\,$ is found (uniqueness again follows in case $\, \rho >0$) we are in the condition of considering the limit as $\, h \,$ goes to zero as well. To this aim, we shall assume that the corresponding initial data converge together with their energies, namely
$$W_{\rho,\nu}(v_{h,0},z_{h,0})  - \lan L(0),(v_h(0),z_h(0)) \ran \rightarrow W_{\rho,\nu}(v_{0},z_{0}) - \lan L(0),(v_0,z_0) \ran. $$
 In this case, it is straightforward to check that the bound \eqref{bound99} is preserved while passing to the limit in $\, h$. Assuming $\, \nu >0$, this entails the possibility of extracting a (not relabeled) subsequence pointwise converging to an energetic solution $\, (v,z): [0,T] \rightarrow\Y_0$. In case $\, \rho >0$, the latter is indeed the unique energetic solution whose existence is stated in Theorem \ref{const2}. In order to check this we briefly comment on relations \eqref{SS}-\eqref{EE}. As for \eqref{SS}, let us fix $\, t \in [0,T]\,$ and any $\, (\overline v, \overline z)\in \Y_0\,$ and exploit the stability of $\, (v_h(t),z_h(t))\,$ in order to get that, for all $\, (\overline v, \overline z) \in \Y$,
\begin{gather}
\W_{\rho,\nu}(v_h(t),z_h(t)) - \lan L(t),(v_h(t),z_h(t)) \ran \nonumber\\
\leq \W_{\rho,\nu}( p_h^\nu( \overline v, \overline z))  - \lan L(t),p_h^\nu (\overline v, \overline z))\ran+\D(z_h - p^\nu_{h,2} (\overline v,\overline z)) . \nonumber
\end{gather}
Hence, the stability of $\, (v(t),z(t))\,$ follows by passing to the limit in $\, h$. As for the upper energy estimate we fix a uniform partition $\, Q^m:=\{s^m_j, \ j=0,\dots,M \ : \ s^m_j= jt/m\}$, exploit the upper energy estimate for $\, (v_h,z_h)$, and get that
\begin{gather}
  \W_{\rho,\nu}(v_h(t),z_h(t)) -   \lan L(t),(v_h(t),z_h(t)) \ran + q(t) + \sum_{j=1}^m \D(z_h(s^m_j) - z_h(s^m_{j-1})) \nonumber\\
 \leq \W_{\rho,\nu}(v_{h,0},z_{h,0}) - \lan L(0),(v_{h,0},z_{h,0}) \ran + q(0) \nonumber \\
- \int_0^t  \lan \dot\ell(s),v_h(s)\ran\, ds- \int_0^t  \lan \dot\ell(s),\overline u(s)\ran\, ds.\nonumber
\end{gather}
It hence suffices to pass to the limit in $\, h \,$ first and then in $\, m
\,$ in order to get the upper energy estimate for $\, (v,z)$. Finally, the
lower energy estimate for $\, (v,z)\,$ follows as above from the upper energy
estimate, stability, uniform convexity of $\, \W_{\rho,\nu}$, and the
continuity of $\, L \,$ and $\, q$. We refer to
\azzurro\cite{MieRou06,MiRoSt06} \nero for a full proof of the above convergence argument. However, we shall remark that no quantitative estimates for the approximations are given. 

\paragraph{Convergence for the time-discretized problem.} Let us consider now the limit as $\, h \,$ goes to $\, 0 \,$ first. Owing to Lemma \ref{conve} we are in the position of establishing a (quantitative) strong convergence result for the corresponding time discretized solutions $\, (v^n,z^n)$. Indeed, one could exhibit some explicit error control which however explodes with $\, n$. Moreover, in the case $\, \nu>0$, since $\, (v^n,z^n)\,$ are uniquely determined, the subsequent limit in $\, n \,$ can be taken exactly as above and the convergence to an energetic solution $\, (v,z)\,$ is ensured.

\paragraph{Joint convergence.} Assume now $\, \nu>0$. Owing to \eqref{bound99} we are of course in the position of passing to the limit with respect to both $\, n \,$ and $\, h \,$ simultaneously in $\, (v^n_h,z^n_h)$. By arguing as above the stability of the limit $\, (v,z)\,$ will follow at once by using the closedness of $\, \S \,$ and the convergence of projections. As for the upper energy estimate, we combine the above exploited techniques and pass to the $\, \liminf\,$ in the following relation (see \eqref{preupper2})
\begin{gather}
\W_{\rho,\nu}(v^n_h(t), z^n_h(t)) -  \lan L(\tau^n(t)),(v^n_h(t),z^n_h(t))\ran + q(\tau^n(t))
+ \Diss_\D( z^n_h,[0,\tau^n(t)])\nonumber\\
 \leq \W_{\rho,\nu}( v_{0,h},z_{0,h}) -  \lan L(0),(v _{0,h},z_{0,h})\ran + q(0)\nonumber\\
-\int_{0}^{\tau^n(t)} \lan \dot \ell(s), v^n_h(s)\ran\, ds  -\int_{0}^{\tau^n(t)} \lan \dot \ell(s),\overline u(s)\ran\, ds.\label{preupper3}
\end{gather}
Once the upper energy estimate is established, the uniform convexity of $\, \W_{\rho,\nu}$ the continuity of $\, L\,$ and $\, q $, and the stability of $\, (v,z)\,$ entail that also the lower energy estimate holds. Namely, $\, (v,z)\,$ is an energetic solution to \eqref{SS}-\eqref{EE} and it is unique as $\, \rho >0$.

Of course, whenever $\, \rho >0 \,$ we would be able to show some convergence of order $\, 1/2\,$ in time. On the other hand, by passing to the limit in time we loose the chance to estimate the error in space (see above). Hence, so far we are not able to provide an explicit space-time error bound for the joint limit procedure.

\section{The limits $\, \boldsymbol \rho\boldsymbol ,\, \boldsymbol \nu\boldsymbol \rightarrow \boldsymbol 0$.}\label{limitrho}

Up to this point, the parameters $\, \rho \,$ and $\, \nu \,$ have been systematically assumed to be fixed throughout the analysis. The limit $\, \nu \rightarrow 0 \,$ is however of some interest since it describes the behavior of the model toward its non-regularized limit. As for $\, \rho \,$ we have to mention that our modeling choice corresponds to the limit situation $\, \rho = 0 \,$. On the other hand the smooth situation $\, \rho > 0 \,$ is better suited for numerical implementation. Moreover, all problems are continuously dependent on data for $\, \rho >0\,$ while energetic evolutions are not known to be unique for $\, \rho =0$. 

In this section we shall discuss the possibility of obtaining suitable asymptotic results for $\, \rho\,$ and (possibly) $\, \nu \,$ going to zero within the constitutive relation, the minimum problem, the incremental problem, and the evolution problem. We will explicitly treat the space approximated case and discuss joint limits of parameters and time and/or space approximations. 

As a general remark, one should notice that the choice $\, \rho=\nu=0\,$ does not affect the well-posedness of the minimum problems since the uniform convexity of the corresponding functionals is preserved, this being true also for space approximations. Secondly, a priori bounds on sequences of solutions (either minimizing, incremental, or energetic) are usually available independently of the parameters. Whenever the compactness of sequences of solutions is obtained, the crucial feature in order to identify the limit of some possibly extracted subsequence is the $\, \Gamma$-convergence (see below) of the approximating functionals $\, W_\rho\,$ (in the zero-dimensional case) and $\, \W_{\rho,\nu} \,$ (in three dimensions). 

\paragraph{$\boldsymbol \Gamma$-convergence issues.} Let us collect here some preliminary remarks on the convergence properties of functions and functionals under consideration. The basic notion in this direction is of course that of $\, \Gamma$-convergence \cite{DeGiorgi-Franzoni75,DeGiorgi-Franzoni79}. The reader is referred to the monographs \cite{Attouch,DalMaso93} for a comprehensive discussion. Let us however recall here that, given a metric space $\, X \,$ and functions $\, g_n,\, g: X \rightarrow (-\infty,+\infty]$, we say that  $\, g_n \rightarrow g\,$ in the sense of $\, \Gamma-$convergence in $\,X\,$ iff
\begin{gather}
  g(x) \leq \liminf_{n \rightarrow +\infty} g_n(x_n) \quad \forall x_n \rightarrow x \ \ \ \text{and} \label{liminf}\\
 \forall x \in X \ \ \text{there exists} \ \ x_n  \rightarrow x \ \ \text{such that} \ \ g(x) \geq \limsup_{n \rightarrow +\infty} g_n(x_n). \label{recovery}
\end{gather}
We shall classically refer to \eqref{liminf} as {\it $\Gamma$-liminf inequality} and to $\, x_n \,$ in \eqref{recovery} as the {\it recovery sequence for $x$}.
Moreover, letting $\, X \,$ be a Banach space, we say that $\, g_n \to g \,$ in the sense of Mosco \cite{Attouch} if $\, g_n \to g\,$ in the sense of $\, \Gamma$-convergence with respect to both the strong and weak topology of~$\, X$.

Let us mention that the issue of the convergence of rate-independent evolution problems under approximation is indeed a crucial one. A general abstract theory of $\, \Gamma$-convergence for rate-independent systems is detailed in \cite{MiRoSt06}. 

Henceforth, we shall refer to the current choice \eqref{sceltaF} and explicitly ask the function $\, f \,$ to be convex and non-decreasing. This entails in particular that $\, F_\rho \rightarrow F \,$ pointwise and non-decreasing. The smoothness of $\, F_\rho \,$ and the latter convergence entail by means of \cite[Thm. 2.40, p. 198]{Attouch} that $\, F_\rho \rightarrow F \,$ in the sense of $\, \Gamma$-convergence in $\, \Rzd$. As a consequence and by using \cite[Thm. 2.15, p. 138]{Attouch}, we have that
\begin{equation}\label{gammaW1}
W_\rho \rightarrow W_0 \quad\text{in the sense of $\,\Gamma$-convergence in} \ \ \Rzn\times \Rzd.
\end{equation}
 As for the three-dimensional situation, let us start by observing that $\, \F_\rho \rightarrow \F \,$ in the sense of $\, \Gamma$-convergence with respect to both the strong and the weak topology in $\, L^2(\Omega; \Rzd)\,$ (namely,   $\, \F_\rho\,$ converges to $\, \F_0 \,$ in the sense of Mosco \cite{Attouch}). This fact follows at once from \cite[Thm. 2.40, p. 198]{Attouch} and the convexity of $\, \F_\rho$. For all $\, \nu>0 \,$ fixed, we readily deduce in a quite similar way that  $\, \F_{\rho,\nu}\,$ converges to $\, \F_{0,\nu} \,$ in the sense of Mosco in $\, H^1(\Omega;\Rzd)$. Let us make precise the latter statement with the following.

 \begin{lemma}[$\boldsymbol \Gamma$-convergence of the inelastic energy]\label{gammaeffe} Let $\, \rho_k \rightarrow \rho\geq 0\,$ and $\, \nu_k \rightarrow \nu \geq 0\, $ be non-increasing. Then $\, \F_{\rho_k,\nu_k} \rightarrow \F_{\rho,\nu}\,$ in the sense of Mosco in $\, H^{j(\nu)}(\Omega;\Rzd)$.
\end{lemma}
   
\begin{proof} The above discussion may be readily extended in order to cover the case $\, \nu_k \to \nu >0$. Let us turn to the situation $\, \nu =0 \,$ and $\, \nu_k >0 \,$ instead. Of course, the $\, \Gamma-$liminf inequality \eqref{liminf} easily follows from the $\, \Gamma$-convergence $\, \F_{\rho_k} \rightarrow \F_\rho \,$ and lower semicontinuity considerations. As for the recovery sequence, letting $\, z \in L^2(\Omega;\Rzd)\,$ be fixed, we shall define $\, z_k \,$ as the unique solution to the singular perturbation problem
$$z_k +\nu_kJz_k = z \quad \text{in} \ \ (H^1(\Omega;\Rzd))',$$
where $\, J :H^1(\Omega;\Rzd) \rightarrow  (H^1(\Omega;\Rzd))'\,$ is the Riesz map. We have that (see, e.g., {\sc Lions} \cite{Lions73})
\begin{gather}
  z_k \rightarrow z \quad \text{strongly in} \ \   L^2(\Omega;\Rzd) \quad \text{and} \quad
\frac{\nu_k}{2}\int_\Omega |\nabla z_k|^2 \rightarrow 0.\nonumber
\end{gather}
Moreover, whenever $\, |z|\leq c_3 \,$ almost everywhere in $\, \Omega$, the same bound holds for all $\, z_k\,$ by the maximum principle. Hence, we readily check that
$$\F_{\rho_k,\nu_k} (z_k) \rightarrow \F_{\rho,0}(z)$$
and the assertion follows. 
   \end{proof}

We shall now turn our attention to the convergence of stored energies and state the following.

\begin{lemma}[$\boldsymbol \Gamma$-convergence of the stored energy]\label{gammavu} Let $\, \rho_k \rightarrow \rho\geq 0\,$ and $\, \nu_k \rightarrow \nu \geq 0\, $ be non-increasing. Then $\, \W_{\rho_k,\nu_k} \rightarrow \W_{\rho,\nu}\,$ in the sense of Mosco in $\, \Y$.
\end{lemma}
We will not provide the reader with a detailed proof. Of course, the argument can be easily reproduced by arguing along the lines of the proof of Lemma \ref{gammaeffe}.

\subsection{Constitutive relation} 

Let us denote by $\, (\epsi,z)_{\rho,\tau}\,$ the incremental solution to the constitutive relation on the partition $\, P:=\{0=t_0<t_1 <\dots<t_{N-1}<t_{N}=T\}  $ with diameter $\, \tau$, namely the right-continuous piecewise constant interpolant on the time partition of the solutions $\, \{(\epsi_{\rho}^i,z_{\rho}^i)\} \,$ to 
$$(\epsi_\rho^i,z_\rho^i)\in\argmin_{(\epsi,z) \in \Rzn \times \Rzd}\big(W_\rho(\epsi,z) - \sigma(t_i):\epsi + D(z - z^{i-1}_\rho) \big)\quad i=1,\dots,N,$$
where $\, \sigma \in W^{1,1}(0,T;\Rzn)\,$ and $\,  (\epsi_{\rho}^0,z_{\rho}^0)=(\epsi_0,z_0)\,$ are given. Moreover, for all $\, \rho \geq 0$, we will denote by $\, (\epsi,z)_{\rho,0}\,$ a solution for the time-continuous constitutive relation. Of course we would be in the position of considering approximating data $\, \sigma_{\rho,\tau}\,$ and $\, (\epsi_0,z_0)_{\rho,\tau}\,$ as well. We limit ourselves to the above situation just for the sake of simplicity. The main result of this subsection is the following.
\begin{theorem}[Convergence for the constitutive relation]\label{CCR}
Let $\, \rho_k \to \rho\geq 0\,$ and $\, \tau_k \to \tau \geq 0\,$ either being constant or converging to $\, 0$. Then, possibly up to the extraction of a subsequence in the case $\, (\rho,\tau)=(0,0)$, we have that
$$(\epsi,z)_{(\rho,\tau)_k}\to (\epsi,z)_{\rho,\tau} \qquad \text{pointwise in} \ \ [0,T].$$
\end{theorem}
Indeed much more is true since the convergence of the component $\, z_{(\rho,\tau)_k} \,$ is uniform and we have convergences also of energies and dissipations.  Moreover, one could consider the limits $\, \rho_k \to \rho >0 \,$ and/or $\, \tau_k \to \tau >0 \,$ as well (which we however believe to be less interesting). We limit ourselves to the above statement for the sake of clarity.

The situation of Theorem \ref{CCR} is described in Figure 1 below where every parameter choice $\, (\rho,\tau)\,$ in the $\, \rho \times \tau \,$ square gives rise to a solution either of the incremental problem (for $\, \tau >0$) or the time-continuous problem ($\tau =0$). Of course this solution is unique if  $\, (\rho,\tau)\not =(0,0)$. Theorem \ref{CCR} entails that all the depicted limits (arrows) can be performed.
\begin{figure}[htbp]
  \centering
\psfrag{r}{$\rho$}
\psfrag{t}{$\tau$}
\psfrag{0}{$(0,0)$}
\psfrag{a}{$a$}
\psfrag{b}{$b$}   
\psfrag{c}{$c$}
\psfrag{d}{$d$}
 \includegraphics[width=.3 \textwidth]{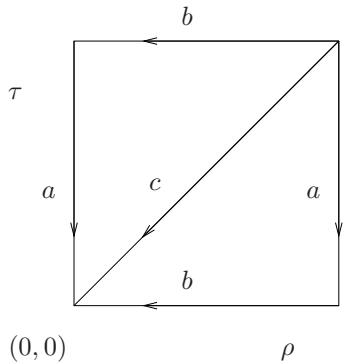}
  \caption{Convergences for the constitutive relation}
  \label{fig:1}
\end{figure}

\begin{proof} By referring to Figure \ref{fig:1}, we shall proceed by discussing limits of type {\it a, b, c,} and {\it d}.

{\it Limits of type a,} namely $\, (\rho,\tau)_k \to (\rho,0)$. These limits follow directly from Theorem \ref{const}.

{\it Limits of type b,} namely $\, (\rho,\tau)_k \to (0,\tau)\,$ with $\, \tau >0$. Since the time partition is fixed, the convergence of the whole sequence $\, (\epsi,z)_{(\rho_k,\tau)}\,$ to the corresponding incremental solution $\, (\epsi,z)_{(0,\tau)}\,$ is ensured by the $\, \Gamma$-convergence of the corresponding energy functionals, their equi-coercivity with respect to $\, \rho$, the continuity of $\, R$, and the continuous dependence of the incremental problem for $\, \rho \geq 0$. 

{\it The limit c,} namely $\, (\rho,\tau)_k \to (0,0)$. Let us now turn to the joint limit. 
 Again, the usual energy and dissipation bounds may be obtained and, by suitably choosing not relabeled subsequences, we find $\,(\epsi,z):[0,T] \rightarrow\Rzn\times \Rzd \,$ such that $\, z_{(\rho,\tau)_k}(t) \rightarrow z(t) \,$ and $\, \epsi_{(\rho,\tau)_k}(t) \rightarrow \epsi(t) \,$ for all $\, t \in [0,T]$. As for to prove the stability of $\, (\epsi(t),z(t)) \,$ we simply need to specialize the closure argument in Theorem \ref{const} by considering the parameter dependence on $\, \rho$. Here, the $\, \Gamma-$convergence \eqref{gammaW1} is again crucial. In particular, let us redefine (see \eqref{stable}), for all $\, \rho \geq 0$,
\begin{gather}
  S_\rho(t):=\Big\{(\epsi,z) \in \Rzn \times \Rzd  \ \ \text{such that}, \ \ \forall (\overline \epsi, \overline z) \in  \Rzn \times \Rzd,\nonumber\\
W_\rho(\epsi,z) - \sigma(t) : \epsi \leq W_\rho(\overline \epsi, \overline z) - \sigma(t) : \overline \epsi+ D(\overline z - z)  
\Big\},  \label{stable99}
\end{gather}
and $\, \S_\rho:= \cup_{t \in [0,T]}(t,S_\rho(t))$. Owing to  the $\, \Gamma-$convergence \eqref{gammaW1} and the continuity of $\, \sigma \,$  we readily check that, for all $\, (t_\rho,\epsi_\rho,z_\rho) \in S_\rho \,$ such that $\, (t_\rho,\epsi_\rho,z_\rho)\,$ converges to $\, (t_0,\epsi_0,z_0)\,$ as $\, \rho \rightarrow 0\,$ one has that $\, (t_0,\epsi_0,z_0) \in S_0$.  As for the upper energy estimate, we readily pass to the $\,\liminf\,$ in the discrete upper equality estimate \eqref{preupper} by means of the $\, \Gamma$-convergence \eqref{gammaW1} and the fact that  $\, W_\rho \rightarrow W_0\,$ pointwise. Finally, the full energy equality follows again from stability.

{\it The limit d,} namely $\, (\rho,0)_k \to (0,0)$. We shall not discuss this limit in detail since it follows easily along the lines of limit {\it c} above.
\end{proof}

\subsection{The minimum problem} 

We investigate for simplicity the situation of fixed data $\, \uD \in H^1(\Omega; \Rz^3)$, $\, \ell \in ( H^1(\Omega; \Rz^3))'$, and $\, \overline z \in L^2(\Omega;\Rzd)$. Of course, some more general situation of parameter-dependent data could be considered as well (see also the forthcoming Lemma \ref{gammaj}). Moreover, let us introduce for the purposes of this section the notation $\, \I_{\rho,\nu}:\Y \rightarrow (-\infty,+\infty] \,$ as
\begin{gather}
\I_{\rho,\nu}(u,z):= \W_{\rho,\nu}(u,z) - \lan \ell,u\ran +\D(z - \overline z)\ \ \ \forall (u,z) \in \Y, \nonumber
\end{gather}
for  all $\, \rho, \, \nu \geq 0$.
Problem \eqref{min} has a unique solution $\, (u,z)_{\rho,\nu} \in \Y(\uD)$ for all given parameters $\, \rho, \, \nu \geq 0$. Moreover, we readily check that $\, \W_{\rho,\nu}((u,z)_{\rho,\nu})\,$ turns out to be bounded independently of $\, \rho\,$ and $\,\nu$. Hence, $\, (u,z)_{\rho,\nu}\,$ is weakly precompact in $\, \Y$.

Moreover, we shall consider the space approximated situation described by the mesh-size $\, h>0$. For the sake of notational simplicity, we reduce ourselves to the oversimplified situation of data independent of $\, h\,$. In particular, we assume $\, \uD \in {\cal U}_h\,$ for $\, h \,$ small enough and define $\, \Y_h(\uD):= \Y_{h,0}+(\uD,0)$. As for the general case, the following discussion has to be restricted to the situation where  convergence \eqref{err} holds for the approximating data $\, \uD_h,\, \ell_h$, and $\, \overline z_h $. Consequently, we will make use of the notation
$$\I_{\rho,\nu,h}(u,z):= \I_{\rho,\nu}(u,z) \quad \text{for} \ \ (u,z) \in \Y_h \ \ \text{and} \ \ +\infty \ \ \text{otherwise in } \ \Y.$$

We shall start by providing the following convergence result.

\begin{lemma}[$\boldsymbol \Gamma$-convergence of $\, \boldsymbol \I_{\boldsymbol \rho\boldsymbol ,\boldsymbol \nu\boldsymbol ,\boldsymbol h}$]\label{gammai} 
Let $\, \rho_k \to \rho \geq 0$, $\, \nu_k \to \nu \geq 0$, and $\, h>0$. Then
  \begin{eqnarray}
    &&\I_{\rho_k,\nu_k} \to \I_{\rho,\nu} \quad \text{in the sense of Mosco in} \ \ \Y,\label{one}\\
  &&\I_{\rho_k,\nu_k,h} \to \I_{\rho,\nu,h} \quad \text{in the sense of Mosco in} \ \ \Y_h.  \label{two}
  \end{eqnarray}
Moreover, let $\, h_k \to 0$. Then 
\begin{equation}
   \I_{(\rho,\nu,h)_k} \to \I_{\rho,\nu} \quad \text{in the sense of Mosco in} \ \ \Y.\label{three}
  \end{equation}
\end{lemma}

\begin{proof} The convergence in \eqref{one} follows directly from Lemma \ref{gammavu} and the strong continuity of $\, \D \,$ in $\, L^2(\Omega,\Rzd)$.

Convergence \eqref{two} is also straightforward. Namely, the $\,\liminf\,$ inequality for weakly converging sequences is immediate and the construction of recovery sequences follows at once from pointwise convergence (recall that $\, {\cal Y}^0_h={\cal Y}^1_h\,$ hence no singular perturbation is needed here).

The full convergence situation of \eqref{three} deserves some comment. Given any $\, (u,z)\in \Y$, we define 
$$(u,z)_{(\rho,\nu,h)_k}:= (q_{h_k}(u),r^{\nu_k}_{h_k}(z)).$$
Owing to the convergence and boundedness properties of the projectors $\, q_{h_k}\,$ and $\,r^{\nu_k}_{h_k}\,$ (see Section \ref{intro}), we readily deduce that $\,(u,z)_{(\rho,\nu,h)_k}\to (u,z) \quad \text{strongly in} \ \ \Y\,$ and 
$$\W_{(\rho,\nu,h)_k}((u,z)_{(\rho,\nu,h)_k})\to \W_{\rho,\nu}(u,z). $$
The $\,\liminf\,$ inequality follows once again from lower semicontinuity.
\end{proof}

The main result of this subsection concerns the possibility of considering (possibly joint) limits in the parameters $\, \rho,\, \nu$, and $\, h \,$ and is graphically represented in Figure 2 below.
\begin{figure}[htbp]
  \centering
\psfrag{r}{$\rho$}
\psfrag{n}{$\nu$}
\psfrag{0}{$(0,0,0)$}
\psfrag{h}{$h$}
 \includegraphics[width= 1 \textwidth]{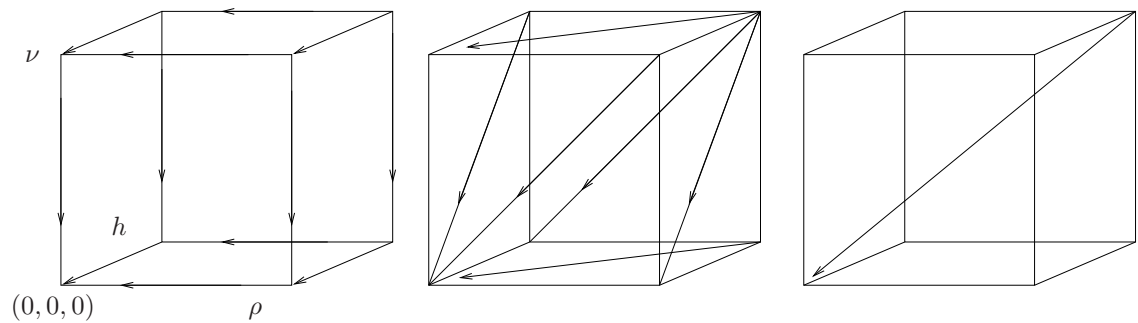}
  \caption{Convergences for the minimum problem}
  \label{fig:2}
\end{figure} 

\begin{theorem}[Convergence for the minimum problem]\label{CMP}
 Let $\, \rho_k \to \rho \geq 0, \, \nu_k \to \nu \geq 0$, and $\, h_k \to h \geq 0 \,$ either being constant or converging to $\, 0$. Then
$$ (u,z)_{(\rho,\nu,h)_k} \to (u,z)_{\rho,\nu,h} \quad \text{weakly in}  \ \ \Y \ \ (\Y_h \ \ \text{if} \ \ h>0).$$
\end{theorem}
 
This result, whose proof is not reported, follows at once from Lemma \ref{gammai} and the equi-coercivity and uniform convexity of the functionals. The limits $\, (\rho,\nu,h) \to (\rho,\nu,0)\,$  where already discussed in detail in Section \ref{incremental}.

\subsection{The incremental problem}\label{IIP} 

We shall extend the latter asymptotics for the minimum problem to the situation of the incremental problem on the fixed partition $\, P:=\{0=t_0<t_1<\dots<t_{N-1}<t_N=T\}$. To this aim let the data $\, \{\uDi\}_{i=0}^N$, $\, \{ \ell^i\}_{i=0}^N\,$ and the initial datum $\, (u^0,z^0)\,$ be suitably given independently of $\, \rho\,$ and $\, \nu \,$ (for simplicity). Then, for all $\, \rho, \, \nu \geq0\,$ we are entitled to solve the incremental problem and find a solution vector $\, \{(u^i_{\rho,\nu},z^i_{\rho,\nu})\}_{i=0}^N$. Now, arguing as above, we easily obtain that $\,\W_{\rho,\nu}(u^i_{\rho,\nu},z^i_{\rho,\nu})\,$ is bounded independently of $\, \rho, \, \nu$, and $\, i$.
For all given $\, \rho, \, \nu \geq 0 $, $\, i =1,\dots,N$, and $\, \overline z \in L^2(\Omega;\Rzd),$ we introduce the functionals
 $\, \J^i_{\rho,\nu}(\cdot,\cdot,\overline z):\Y \rightarrow (-\infty,+\infty] \,$ as
\begin{gather}
\J^i_{\rho,\nu}(u,z,\overline z):= \W_{\rho,\nu}(u,z) - \lan \ell^i,u\ran +\D(z - \overline z)\quad \forall (u,z) \in \Y \nonumber
\end{gather}
Moreover, possibly taking into account the space-approximated situation, one would need to introduce space approximated data  $\, \{\uDi_h\}_{i=0}^N$, $\, \{ \ell^i_h\}_{i=0}^N\,$ and the initial datum $\, (u_h^0,z_h^0)$. Let us however restrict ourselves to the (over)simplified situation where the latter can be assumed to be independent of $\, h $. For all $\, \rho,\, \nu\geq 0$, $h >0$, $\, i =1,\dots,N$, and $\, \overline z \in L^1(\Omega;\Rzd)$, we shall make use of the functionals $\, \J^i_{\rho,\nu,h}(\cdot,\cdot,\overline z):\Y \rightarrow (-\infty,+\infty] \,$ defined as$$\J^i_{\rho,\nu,h}(u,z,\overline z):=\J^i_{\rho,\nu}(u,z,\overline z)\quad \text{if} \ \ (u,z)\in \Y_h \ \ \text{and} \ \ +\infty \ \ \text{otherwise}.$$

Let us start from the following $\, \Gamma$-convergence result.

\begin{lemma}[$\boldsymbol \Gamma$-convergence of $\, {\boldsymbol \J}^{\boldsymbol i}_{\boldsymbol \rho\boldsymbol ,\boldsymbol \nu\boldsymbol ,\boldsymbol h}$]\label{gammaj}
Let $\, \rho_k \to \rho \geq 0$, $\, \nu_k \to \nu \geq 0$, and $\, h>0$. Moreover, let $\, \overline z_k \to \overline z \,$ strongly in $\, L^1(\Omega;\Rzd)$. Then, for all $\, i =1, \dots,N$,
  \begin{eqnarray}
    &&\J^i_{\rho_k,\nu_k}(\cdot,\cdot,\overline z_k) \to \J^i_{\rho,\nu}(\cdot,\cdot,\overline z) \quad \text{in the sense of Mosco in} \ \ \Y, \label{one2}\\
  &&\J^i_{\rho_k,\nu_k,h}(\cdot,\cdot,\overline z_k) \to \J^i_{\rho,\nu,h}(\cdot,\cdot,\overline z) \quad \text{in the sense of Mosco in} \ \ \Y_h.  \label{two2}
  \end{eqnarray}
Moreover, let $\, h_k \to 0$. Then, for all $\, i =1, \dots,N$,
\begin{equation}
   \J^i_{(\rho,\nu,h)_k}(\cdot,\cdot,\overline z_k) \to \J^i_{\rho,\nu}(\cdot,\cdot,\overline z) \quad \text{in the sense of Mosco in} \ \ \Y.\label{three2}
  \end{equation}
\end{lemma}

We are not reporting here the proof of the latter lemma for the sake of brevity. Indeed, the argument may be easily adapted from that of Lemma \ref{gammai} by exploiting the strong continuity of $\, \D \,$ in $\, L^1(\Omega;\Rzd)$, its lower semicontinuity in $\, L^2(\Omega;\Rzd)$, and the triangle inequality \eqref{triangle}.

By using Lemma \ref{gammaj} and denoting by $\, (u,z)_{\rho,\nu}\,$ and $\,  (u,z)_{\rho,\nu,h}\,$ the incremental solutions related to the parameter choice $\, (\rho,\nu)\,$ and, possibly, the space approximation, the main result of this subsection reads as follows.

\begin{theorem}[Convergence for the incremental problem for $\, \boldsymbol \nu \boldsymbol >\boldsymbol 0$]\label{CIP} Let $\, \nu >0\,$ be fixed and $\, \rho_k \to \rho$, and $\, h_k \to h \geq 0 \,$ either being constant of converging to $\,0$. Then, for all $\, t \in [0,T]$,
$$(u(t),z(t))_{\rho_k,\nu,h_k} \to (u(t),z(t))_{\rho,\nu,h} \quad \text{strongly in} \ \ \Y.$$
\end{theorem}

Of course, we would be in the position of considering the case $\, \nu_k \rightarrow \nu$, $\, \rho_k \to \rho>0$, and/or $\, h_k \to h>0\,$ as well. We however restrict to the above situation for the sake of clarity.

  Lemma \ref{gammaj} entails the convergence of the incremental solutions as soon as the strong convergence of $\, z_{\rho_k,\nu}\,$ or $\, z_{\rho_k,\nu,h_k}\,$ in $\, L^1(\Omega;\Rzd)\,$ is ensured. In order to obtain the latter from the boundedness of energy through compactness we are forced once again to restrict our attention to the case $\, \nu>0$. The proof of Theorem \ref{CIP} follows then by simply taking steps in $\, i$.

\subsection{The evolution problem} 

Owing to the latter discussion on the incremental problem (see Lemma \ref{gammaj}), we shall restrict ourselves to the situation $\, \nu >0 \,$ from the very beginning (note that existence is not known for $\, \nu=0$). For all $\, \rho,\, h \geq 0$, let us denote by  $\, (v,z)_\rho :[0,T]\rightarrow \Y_{0} \,$ and $\, (v,z)_{\rho,h} :[0,T]\rightarrow \Y_{0,h} \,$ the solutions to the corresponding energetic formulations for $\, h=0 \,$ and $\, h>0\,$ (here and in what follows we have assumed the data $\, L,\, q$, and the initial datum $\, (v^0,z^0)\,$ to be fixed independently of all approximations). The latter solutions are known to exists and turn out to be unique for $\, \rho >0$. Moreover, let $\, (v,z)_{\rho,\tau} \,$ and $\, (v,z)_{\rho,\tau,h} \,$ denote the unique incremental solutions to the problem on a given partition with diameter $\, \tau$.

A variety of convergence results for  $\, (v,z)_\rho,\, (v,z)_{\rho,h},\, (v,z)_{\rho,\tau}$, and $\, (v,z)_{\rho,\tau,h} \,$ have already been obtained. This subsection will complement the above discussions and complete the picture of convergence results for the time-continuous evolution problem. In particular, as soon as $\, \nu >0\,$ is fixed, we are entitled to take (possibly joint) limits in $\,(\rho,\tau,h)\,$ as it is graphically depicted in Figure 3 below.
\begin{figure}[htbp]
  \centering
\psfrag{a}{$a$}
\psfrag{b}{$b$}
\psfrag{c}{$c$}
\psfrag{d}{$d$}
\psfrag{e}{$e$}
\psfrag{f}{$f$}
\psfrag{r}{$\rho$}   
\psfrag{t}{$\tau$}
\psfrag{0}{$(0,0,0)$}
\psfrag{h}{$h$}
 \includegraphics[width= 1 \textwidth]{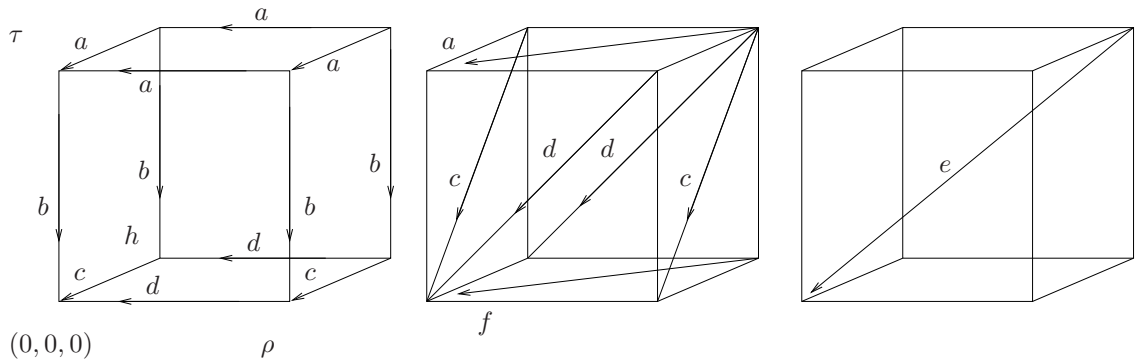}
  \caption{Convergences for the evolution problem $(\nu >0)$}
  \label{fig:3}
\end{figure} 

The main result of this subsection reads as follows.

\begin{theorem}[Convergence for the evolution problem for $\,\boldsymbol \nu \boldsymbol >\boldsymbol 0$]\label{EP}
   Let $\, \nu >0\,$ be fixed and $\, \rho_k \to \rho$, $\, \tau_k \to \tau\geq 0$, and $\, h_k \to h \geq 0 \,$ either being constant of converging to $\,0$. Then, possibly extracting not-relabeled subsequences if $\, (\rho,\tau)=(0,0)$, for all $\, t \in [0,T]$,
$$(v(t),z(t))_{(\rho,\tau,h)_k} \to (v(t),z(t))_{\rho,\tau,h} \quad \text{strongly in} \ \ {\cal Y}^\nu_0.$$
\end{theorem}

\begin{proof}[Sketch of the proof.] Referring to Figure 3, let us start by observing that the limits of type $\, a\,$ and $\, b\,$ were already obtained in Theorem \ref{CIP} and Theorem \ref{const2}, respectively. Moreover, the limits of type $\, c \,$ have been discussed at the end of Section \ref{evoevo}.

{\it Limits of type d.} This limits can be established by simply adapting to the current three-dimensional situation the argument of Theorem \ref{CCR}. In case $\, h>0$, the latter adaptation is even simplified by finite-dimensionality and the convergence result would hold for $\, \nu =0 \,$ as well.

{\it The limit e.} By suitably extracting (not-relabeled) subsequences we readily find $\, (v,z): [0,T]\to {\cal Y}^\nu_0\,$ such that, for all $\, t \in [0,T]$, 
\begin{eqnarray}
&&(v(t),z(t))_{(\rho,\tau,h)_k} \to (v(t),z(t)) \quad \text{weakly in}\ \  {\cal Y}^\nu_0,\nonumber\\
&& z_{(\rho,\tau,h)_k}(t)\to z(t) \quad \text{strongly in}\ \ L^2(\Omega;\Rzd).
\end{eqnarray}
Hence, we are left to prove that indeed $\, (v,z)\,$ is a solution of the evolution problem, i.e., check for the stability condition \eqref{S} and the energy equality \eqref{E}.

As for the former, we exploit Lemma \ref{gammavu} and, for all $\, (\overline v, \overline z) \in {\cal Y}^\nu_0$, by letting $\, (\overline v, \overline z)_k := (q_{h_k}(\overline v), r_{h_k}^\nu(\overline z))\,$ we check that 
\begin{gather}
  \W_{0,\nu}(v(t),z(t)) - \lan L(t),(v(t),z(t))\ran \nonumber\\
\leq \liminf_{k\rightarrow +\infty}\Big( \W_{\rho_k,\nu}((v(t_{\tau_k}),z(t_{\tau_k}))_{(\rho,\tau,h)_k}) - \lan L(t_{\tau_k}),((v(t_{\tau_k}),z(t_{\tau_k}))_{(\rho,\tau,h)_k}) \Big)\nonumber\\
\leq \liminf_{k\rightarrow +\infty}\Big( \W_{\rho_k,\nu}((\overline v, \overline z)_k) - \lan L(t_{\tau_k}),((\overline v, \overline z)_k) +\D(\overline z_h - z_{(\rho,\tau,h)_k})\Big)\nonumber\\
= \W_{0,\nu}(\overline v, \overline z) - \lan L(t),(\overline v, \overline z) \ran + \D(\overline z - z(t))\nonumber
\end{gather}
where we used some obvious notation for the point $\, t_{\tau_k}\,$ on the time-partition of diameter $\, \tau_k \,$ such that $\, 0 \leq t -t_{\tau_k}  < \tau_k$, Lemma \ref{gammavu}, the stability of $\, (v,z)_{(\rho,\tau,h)_k}\,$ at time $\, t_{\tau_k}$, and the strong continuity of $\, \D \,$ in $\, L^2(\Omega;\Rzd)$.

The upper energy estimate (and hence \eqref{E}) follows by simply passing to the $\, \liminf\,$ as $\, (\rho,\tau,h)_k \to (0,0,0)\,$ in the discrete upper energy estimate \eqref{preupper2}.
 
{\it The limit f.} This limit can be obtained along the same lines of limit $\, e\,$ above, the argument being even simplified by the fact that here $\, \tau_k=0\,$ and the upper energy estimate follows by passing to the $\,\liminf\,$ as $\, (\rho,h)_k \to (0,0)\,$ in the time-continuous upper energy estimate \eqref{upper3}.
\end{proof}

\def\cprime{$'$} \def\cprime{$'$}



\end{document}